\documentclass[a4paper]{amsart}
\usepackage{amssymb,amsmath,amsthm}

\usepackage[T1]{fontenc}
\usepackage[utf8]{inputenc}

\usepackage{microtype}
\usepackage{setspace}

\usepackage[unicode=true,pdfusetitle,bookmarks=true,bookmarksnumbered=false,bookmarksopen=false,
breaklinks=false,pdfborder={0 0 1},backref=false,colorlinks=false]{hyperref}

\usepackage[nameinlink,capitalize]{cleveref}

\usepackage{comment}
\usepackage[colorinlistoftodos]{todonotes}

\usepackage{enumitem}

\usepackage{mathtools} 

\usepackage{comment}

\theoremstyle{plain}
\newtheorem{thm}{\protect\theoremname}[section]
\newtheorem{prop}[thm]{\protect\propositionname}
\newtheorem{lem}[thm]{\protect\lemmaname}
\newtheorem{cor}[thm]{\protect\corollaryname}
\newtheorem{fact}[thm]{\protect\factname}

\newtheorem{obs}[thm]{\protect\observationname}

\theoremstyle{definition}
\newtheorem{defn}[thm]{\protect\definitionname}
\newtheorem{exmp}[thm]{\protect\examplename}
\newtheorem{question}[thm]{\protect\questionname}

\theoremstyle{remark}
\newtheorem{rem}[thm]{\protect\remarkname}

\newtheorem{claim}[thm]{\protect\claimname}

\crefname{thm}{Theorem}{Theorems}
\crefname{lem}{Lemma}{Lemmas}

\providecommand{\theoremname}{Theorem}
\providecommand{\propositionname}{Proposition}
\providecommand{\lemmaname}{Lemma}
\providecommand{\corollaryname}{Corollary}
\providecommand{\factname}{Fact}
\providecommand{\conjecturename}{Conjecture}
\providecommand{\observationname}{Observation}
\providecommand{\definitionname}{Definition}
\providecommand{\examplename}{Example}
\providecommand{\questionname}{Question}
\providecommand{\remarkname}{Remark}
\providecommand{\notationname}{Notation}
\providecommand{\claimname}{Claim}
\providecommand{\casename}{Case}

\newcommand{\bb}[1]{\mathbb{#1}}
\newcommand{\cal}[1]{\mathcal{#1}}

\newcommand{\N}{\bb{N}}
\newcommand{\Z}{\bb{Z}}
\newcommand{\Q}{\bb{Q}}

\newcommand{\tp}{\operatorname{tp}}
\newcommand{\eq}{\operatorname{eq}}

\newcommand{\cM}{\cal{M}}
\newcommand{\cN}{\cal{N}}
\newcommand{\cZ}{\cal{Z}}

\newcommand{\cP}{\cal{P}}
\newcommand{\cG}{\cal{G}}
\newcommand{\cD}{\cal{D}}

\newcommand{\curlyb}[1]{ \left\{ #1 \right\} }
\newcommand{\set}[2]{ \left\{ #1 \, : \, #2 \right\} }

\newcommand{\roundedb}[1]{ \left( #1 \right) }
\newcommand{\squareb}[1]{ \left[ #1 \right] }

\newcommand{\pipes}[1]{ \left| #1 \right| }

\newcommand{\dcl}{\operatorname{dcl}}
\newcommand{\acl}{\operatorname{acl}}

\newcommand{\Th}[1]{\operatorname{Th} \roundedb{ #1 }}

\newcommand{\spn}[2][]{\operatorname{span}_{#1} \roundedb{ #2 }}
\newcommand{\spnplus}[1]{\operatorname{span}^{+} \roundedb{ #1 }}

\newcommand{\Ind}[2]{ #1\setbox0=\hbox{$#1x$}\kern\wd0\hbox to 0pt{\hss$#1\mid$\hss} \lower.9\ht0\hbox to 0pt{\hss$#1\smile$\hss}\kern\wd0
}

\newcommand{\Notind}[2]{ #1\setbox0=\hbox{$#1x$}\kern\wd0\hbox to 0pt{\mathchardef \nn="3236\hss$#1\nn$\kern1.4\wd0\hss}\hbox to 0pt{\hss$#1\mid$\hss}\lower.9\ht0 \hbox to 0pt{\hss$#1\smile$\hss}\kern\wd0
}

\newcommand{\ind}{\mathop{\mathpalette\Ind{}}}

\title[Stable reducts of Abelian groups]{Stable reducts of elementary extensions of Presburger arithmetic}

\author[A. Alouf]{Eran Alouf}
\address{Einstein Institute of Mathematics, Hebrew University of Jerusalem,
91904, Jerusalem Israel.}
\email{Eran.Alouf@mail.huji.ac.il}
\author[A. Fornasiero]{Antongiulio Fornasiero}
\address{Dipartimento di Matematica e Informatica `Ulisse Dini' (DIMAI),
Università degli Studi di Firenze, 50134 Firenze (FI), Italy}
\email{antongiulio.fornasiero@unifi.it}
\urladdr{https://sites.google.com/site/antongiuliofornasiero}
\author[I. Kaplan]{Itay Kaplan}
\address{Einstein Institute of Mathematics, Hebrew University of Jerusalem,
91904, Jerusalem Israel.}
\email{kaplan@math.huji.ac.il}

\thanks{The third author would like to thank the Israel Science Foundation (grants no. 1254/18 and 804/22).}



\begin{document}

\begin{abstract}
Suppose $N$ is elementarily equivalent to an archimedean ordered abelian group $\left( G,+,< \right)$ with small quotients (for all $1 \leq n < \omega$, $[G: nG]$ is finite). Then every stable reduct of $N$ which expands $\left( G,+ \right)$  (equivalently every reduct that does not add new unary definable sets) is interdefinable with $\left( G,+ \right)$. This extends previous results on stable reducts of $\left(\Z, +, <\right)$ to (stable) reducts of elementary extensions of~$\Z$. In particular this holds for $G = \Z$ and $G = \Q$. As a result we answer a question of Conant from 2018.

This result is a corollary of a more general statement about expansions of weakly-minimal 1-based expansions of abelian groups with small quotients preserving the $\acl$ operator.
\end{abstract}

\maketitle

\section{Introduction}
The classification of reducts and expansions of various structures is one of the most fundamental problems in model theory. In particular, classical structures expanding $\left( \Z,+,0,1 \right)$ such as Presburger arithmetic (i.e., $\left( \Z,+,0,1,< \right)$), their reducts and expansions (with various properties), have been the center of many recent papers. A remarkable result of Conant in this direction is the following:

\begin{fact}[\cite{Conant2018_no_intermediate}]
\label{fact_conant_no_intermediate_structures_order}Suppose that $\cZ$ is an expansion of $\roundedb{\Z,+,0,1}$ and a reduct of $\roundedb{\Z,+,0,1,<}$. Then $\cZ$ is interdefinable with either $\roundedb{\Z,+,0,1}$ or $\roundedb{\Z,+,0,1,<}$.
\end{fact}

The proof used a ``hands on'' approach to analyze definable sets in Presburger arithmetic. 
One of Conant's motivations was to answer a question of Marker from 2011 asking about the case where $\cZ$ is \emph{stable}. However, this special case was already solved by Conant and Pillay in \cite{ConantPillay2018} where they showed that $\roundedb{\Z,+,0,1}$ has no proper stable expansions of finite dp-rank. 

In \cite{Alouf_dElbee_19}, d'Elb\'{e}e and the first author used the aforementioned result from \cite{ConantPillay2018} to give an alternative (and arguably simpler) proof for \cref{fact_conant_no_intermediate_structures_order} (see \cite[Theorem 5.14]{Alouf_dElbee_19}). 
They also proved an analogous result for $\roundedb{\Z,+,0,1,\preceq_{p}}$, where $p$ is prime and $\preceq_{p}$ is the partial order associated with the $p$-adic valuation $v_p$ (i.e., $a \preceq_{p} b$ holds if and only if $v_p \roundedb{a} \leq v_p \roundedb{b}$):

\begin{fact}[{\cite[Corollary 1.12]{Alouf_dElbee_19}}]
\label{fact_alouf_dElbee_no_intermediate_structures_p_adic}
Suppose that $\cZ$ is an expansion of $\roundedb{\Z,+,0,1}$ and a reduct of $\roundedb{\Z,+,0,1,\preceq_{p}}$. Then $\cZ$ is interdefinable with either $\roundedb{\Z,+,0,1}$ or $\roundedb{\Z,+,0,1,\preceq_{p}}$.
\end{fact}

In \cite{Alouf_dElbee_19} it was also shown that both \cref{fact_conant_no_intermediate_structures_order} and \cref{fact_alouf_dElbee_no_intermediate_structures_p_adic} fail when $\roundedb{\Z,+,0,1,<}$ and $\roundedb{\Z,+,0,1,\preceq_{p}}$, respectively, are replaced with proper elementary extensions of $\roundedb{\Z,+,0,1,<}$ and $\roundedb{\Z,+,0,1,\preceq_{p}}$. For proper elementary extensions of $\roundedb{\Z,+,0,1,\preceq_{p}}$, it was shown that there are both stable and unstable intermediate structures (see \cite[Propositions 6.1 and 6.2]{Alouf_dElbee_19}). However, for proper elementary extensions of $\roundedb{\Z,+,0,1,<}$, it was only shown that there are unstable intermediate structures (see \cite[Proposition 6.3]{Alouf_dElbee_19}), and the question remained whether there are also stable intermediate structures.
The work leading to this paper began as an attempt to answer this question, and indeed, we give a negative answer as a special case of \cref{main_theorem_2_intro}.

To formalize our results, we distinguish between the notions of expansion/reduct and $0$-expansion/$0$-reduct, see \cref{def_reduct_expansion}.
Our first main theorem is as follows:

\begin{thm}[\cref{main_theorem}]
\label{main_theorem_1_intro}
Let $M$ be a weakly-minimal and 1-based structure expanding an abelian group $\left( G,+ \right)$ with small quotients (i.e., for all $1 \le n < \omega$, $\squareb{G:nG}$ is finite).
Let $N$ be a $0$-expansion of $M$ such that:
\begin{enumerate}
\item $N$ is $\omega$-saturated.
\item Every unary subset $X \subseteq G$ that is definable in $N$ is also definable in $M$.
\item For every $A \subseteq G$, $\acl_M \roundedb{A} = \acl_N \roundedb{A}$.
\end{enumerate}
Then $M$ and $N$ are interdefinable.
\end{thm}

A key ingredient in the proof of \cref{main_theorem_1_intro} is the following preservation lemma for group homomorphisms, which may be of independent interest:

\begin{lem}[\cref{lem_homomorphism_preservation_small_quotients}]
Suppose that $N$ is an $\omega$-saturated structure $0$-expanding a structure $M$ that expands an abelian group $\left( G, + \right)$ with small quotients.
Let $f : G \to G$ be a group homomorphism definable in $N$. If there exists a finite set $F \subseteq G$ such that for every $g\in G$, $f\left(g\right) \in \acl_{M}\left(Fg\right)$, then $f$ is definable in $M$.

In particular, if  $\dcl_{N} \subseteq \acl_{M}$, then every group homomorphism $f : G \to G$ definable in $N$ is definable in $M$.
\end{lem}

We also prove \cref{lem_homomorphism_preservation_R_groups}, another version of the same idea that holds for certain groups that are not necessarily abelian.

The assumption that the group has small quotients is essential in \cref{main_theorem_1_intro}, see \cref{exa:need small quotients}. It would be interesting to find more instances of such a result, and a natural question (see \cref{que:fields}) is what happens if we replace the group with a field. A related result was indeed proved by Hrushovski, where $M$ is assumed to be a pure algebraically closed field, see \cref{hrushovski_theorem_about_strongly_minimal_expansions_of_fields}.

The second main theorem is an application of \cref{main_theorem_1_intro} to ordered abelian groups:

\begin{thm}[\cref{thm_application_small_quotients}]
\label{main_theorem_2_intro}
Let $N = \roundedb{G,+,<}$ be elementarily equivalent to an archimedean ordered abelian group with small quotients, and let $M$ be a reduct of $N$ that expands $\roundedb{G,+}$. Suppose that one of the following holds:
\begin{enumerate}

\item $M$ is stable, or

\item every unary subset $D \subseteq G$ that is definable in $M$, is also definable in $\roundedb{G,+}$.
\end{enumerate}

Then $M$ is interdefinable with $\roundedb{G,+}$.
\end{thm}

As part of the proof, we show that in the settings of \cref{main_theorem_2_intro}, the assumptions (1) and (2) are equivalent (see \cref{prop_equiv_between_stability_and_no_new_unary_def_subsets_and_no_new_unary_def_subsets_in_monster}). 

In particular, \cref{main_theorem_2_intro} holds if, instead of small quotients, we make the stronger assumption that $G$ has finite rank as a torsion-free abelian group (see \cref{cor_application_finite_rank}).

As an application of the proof we prove that in the context of \cref{main_theorem_2_intro}, if $M$ is a proper expansion of $\left( G,+ \right)$ then $M$ defines the order on some infinite interval, see \cref{cor_defining_the_order_on_an_interval}. This answers positively a question of Conant, see \cite[Question 1.5]{Conant2018_no_intermediate}.

\subsection{Structure of the paper}
In \cref{sec:preliminaries} we establish all the facts and notations needed to prove the main results. In particular, we discuss expansions and reducts in \cref{subs:exapnsions and reducts}, facts about groups in \cref{subs:groups}, weak minimality and 1-basedness in \cref{subs:weak minimality}, and imaginaries in \cref{subs:weak and geometric EI} and Archimedean ordered abelian groups in \cref{subs:archimedean OAGs}. \cref{sec:preservation of hom} is devoted to \cref{lem_homomorphism_preservation_small_quotients} mentioned above. \cref{sec:main theorem} handles \cref{main_theorem_1_intro} and \cref{sec:elementary extentions} deals with \cref{main_theorem_2_intro}.

\section{Preliminaries} \label{sec:preliminaries}

\subsection{Expansions and reducts}\label{subs:exapnsions and reducts}
Here we make precise what we mean when discussing expansions and reducts, and collect some observations.
In this paper, when we say that a subset $X$ in a structure $M$ is definable without specifying a parameter set, we mean that $X$ is definable with parameters from $M$.  

\begin{defn}
\label{def_reduct_expansion}
Let $M$ and $N$ be two structures with the same universe, and let $A$ be a subset of this universe.
\begin{enumerate}
\item The structure $M$ is an \emph{$A$-reduct} of $N$ if whenever $X \subseteq M^{k}$ is $\emptyset$-definable in $M$, it is also $A$-definable in $N$. We also say that $N$ is an \emph{$A$-expansion} of $M$.
\item The structures $M,N$ are \emph{$A$-interdefinable} iff $M$ is an $A$-reduct of $N$ and vice versa.
\item When $A$ is the entire universe, we simply say \emph{reduct}, \emph{expansion} and \emph{interdefinable}. When $A=\emptyset$ we also write \emph{$0$-reduct}, \emph{$0$-expansion} and \emph{$0$-interdefinable}.
\end{enumerate}
\end{defn}

\begin{exmp}
Let $N = \roundedb{G,+,<}$ be a proper elementary extension of $ \roundedb{\Z,+,<}$. Let $a \in G$ be a positive and infinite element, and $M := \roundedb{G,+,\squareb{0,a}}$. Then $M$ is a reduct of $N$, but it is not a $0$-reduct.  
\end{exmp}

In this paper, we focus only on definable sets, not specific languages. Therefore, whenever a structure $N$ is a $0$-expansion of a structure $M$, we can (and will) assume that the language of $N$ contains the language of $M$.

\begin{obs}
\label{interdefinability_passes_to_elementary_substructures}
Let $M$ be a structure and $N$ be an expansion of $M$. Let $N'$ be an elementary extension of $N$, and let $M'$ be the reduct of $N'$ to the language of $M$. If $M'$ and $N'$ are interdefinable, then $M$ and $N$ are interdefinable.
\end{obs}

\begin{obs}
\label{obs_not_adding_unary_subsets_enough_to_check_for_saturated_enough}
Let $M$ be an $L$-structure, and let $N$ be an expansion of $M$ which is $\left|L\right|^+$-saturated. Let $N'$ be an elementary extension of $N$, and let $M'$ be the reduct of $N'$ to~$L$. Suppose that every unary subset $X \subseteq N$ that is definable in $N$ is also definable in~$M$. Then, every unary subset $X \subseteq N'$ that is definable in $N'$ is also definable in~$M'$.
If $N$ is a $0$-expansion of $M$ it is enough to assume that $N$ is $\omega$-saturated.
\end{obs}

\begin{defn}
Let $M$ be a structure, and let $A \subseteq M$.
A formula $\phi\left(x\right)$ over $A$ is algebraic if $\phi\left(M\right)$ is finite. We let $\acl_{M}$ be the \emph{algebraic closure operator}.
Namely, $\acl_{M}:\cP\left(M\right)\to\cP\left(M\right)$ and $\acl_{M}\left(A\right)$ is the union of all realizations of algebraic formulas over $A$.

Similarly, we define $\dcl_{M}$, the \emph{definable closure operator}, as the function $\dcl_{M}: \cP\left(A\right)\to\cP\left(A\right)$ such that $\dcl_{M}\left(A\right)$ is the union of all realizations
of algebraic formulas of size 1 over $A$.

We omit $M$ if it is clear from the context.
\end{defn}

\begin{obs}
\label{obs_equality_of_acl_operators_enough_to_check_for_saturated_enough}
Let $M$ be a structure with language $L$, and let $N$ be a $0$-expansion of $M$ that is $\omega$-saturated. Let $N'$ be an elementary extension of $N$, and let $M'$ be the reduct of $N'$ to $L$. Suppose that $\acl_M = \acl_N$. Then $\acl_{M'} = \acl_{N'}$.
\end{obs}

\begin{proof}
It is enough to show that for every finite tuple $c$ in $N'$, we have $\acl_{M'}\left(c\right) = \acl_{N'}\left(c\right)$. Thus, let $c$ be a finite tuple in $N'$. Since $N'$ is a $0$-expansion of $M'$, we have $\acl_{M'}\left(c\right) \subseteq \acl_{N'}\left(c\right)$. 

For the other direction, let $e \in \acl_{N'}\left(c\right)$. Let $e', c' \in N$ be such that $ec \equiv e'c'$ (exists by saturation). Clearly, $e' \in \acl_N \left( c' \right)$. Thus, $e' \in \acl_M \left( e' \right)$. Since this information is part of the type $\tp \left( e'c' \right) = \tp \left( ec \right)$, it follows that $e \in \acl_{M'} \left( c \right)$. 
%
\end{proof}

\pagebreak[3]

\subsection{Groups}\label{subs:groups}

\begin{defn}
We say that a group $G$ has \emph{unbounded exponent} if for all $1\le n < \omega$, $G^n \neq \{1\}$.
\end{defn}

\begin{defn}
\label{def_small_quotients_and_torsion}
Let $\left(G,+\right)$ be an abelian group.
\begin{enumerate}
\item We say that $G$ has \emph{small quotients} if for all $1\le n < \omega$, $nG$ has finite index in $G$.
\item We say that $G$ has \emph{small torsion} if for all $1\le n < \omega$, the set $\set{g \in G}{ng=0}$ is finite\footnote{They are also called almost torsion-free in the literature: see \cite{tkach}.}.
\end{enumerate}
\end{defn}

\begin{rem}
\,  

\begin{enumerate}
\item If $G$ is torsion-free, then it has small torsion, but it does not necessarily have small quotients (e.g., $\Z^{\omega}$).
\item If $G$ is infinite and has either small quotients or small torsion, then $G$ has unbounded exponent.
\end{enumerate}
\end{rem}

\begin{defn}
\label{def_R-group}
We say that a group $G$ is an \emph{$R$-group}\footnote{For a discussion on the name, see \cite[page 2]{Baumslag}.}  if it has unique roots: for all $1 \le n < \omega$ and $x,y \in G$, if $x^{n}=y^{n}$ then $x=y$. 
\end{defn}

\begin{rem}
If $G$ is an $R$-group, then $G$ is torsion-free.
\end{rem}

Examples of $R$-groups include torsion-free nilpotent groups, see \cite[Corollary 2.5]{MR1670603}. See also \cite{MR0122859}.

\begin{rem}
\label{small_quotients_or_torsion_are_elementary_properties}
The properties of having small quotients and having small torsion are  elementary properties of abelian groups. That is, if $G_1$ and $G_2$ are abelian groups such that $G_1 \equiv G_2$, then $G_1$ has this property iff $G_2$ has it. 
Similarly, being an $R$-group is an elementary property of groups.
\end{rem}

In the context of finite $U$-rank, having small torsion and having small quotients amounts to the same thing. For more on $U$-rank, see e.g., \cite[Section 8.6]{TentZieglerBook}.
\begin{lem}
\label{for_ab_group_of_fin_U_rk_small_torsion_equiv_small_quotients}
Let $\left(G,+\right)$ be an abelian group of finite $U$-rank. 
$G$ has small torsion if and only if $G$ has small quotients.
\end{lem}

\begin{proof}
By the Lascar inequalities for groups (see \cite[Corollary 6.3]{MR1827833}), for each $ n < \omega$, $U \left( G \right) = U \left( G/nG \right) + U \left( nG \right)$.
Let $\phi_{n} : G \to G$ be the homomorphism defined by $\phi_{n} \left( g \right) := ng$.
Then, again by the Lascar inequalities for groups (and the fact that $U$-rank is preserved under $0$-definable bijections by  the Lascar inequalities for types) 
$U \left( G \right) = U \left( \ker \phi_{n} \right) + U \left( nG \right)$. Hence $ U \left( G/nG \right) = U \left( \ker \phi_{n} \right)$.
Finally, $G$ has small torsion iff $ U \left( \ker \phi_{n} \right) = 0$ for all $n$, and $G$ has small quotients iff $U \left( G/nG \right) = 0$ for all $n$.
\end{proof}

Recall that the \emph{divisible hull} of a torsion-free abelian group $A$ is the tensor product $\hat{A} := \Q \otimes_{\Z} A$. $\hat{A}$ is a divisible abelian group, and $A$ canonically embeds into $\hat{A}$ via $a \mapsto  1 \otimes a$. Abusing notation, we consider $A$ as a subgroup of $\hat{A}$.

Recall that the rank of a torsion-free abelian group is the dimension of its divisible hull as a vector space over $\Q$.

\begin{fact}[{\cite[Theorem 0.1]{Arnold:82}}]
\label{fact_torsion_free_abelian_group_of_finite_rank_has_small_quotients}
A torsion-free abelian group of finite rank has small quotients.
\end{fact} 
For the convenience of the reader we provide a proof (which is different from the one in \cite[Theorem 0.1]{Arnold:82}).
\begin{proof}
We will prove a stronger statement: 
\begin{claim}
If $A$ is a subgroup of $\Q^d$, then, for every $1 \leq m \in \N$,
$\pipes{A/mA} \leq m^d$.
\end{claim}
\begin{proof}[Proof of claim]
If $A$ is finitely generated, 
then $A \cong \Z^{d'}$ for some $d' \leq d$, and the claim follows.
In general, let $b_1, \dotsc, b_{\ell} \in A$ and let $B$ be the subgroup of $A$ generated by $b_1, \dotsc, b_{\ell}$.
Then, 
\[
\pipes{\curlyb{b_1/ mA, \dotsc, b_{\ell}/ mA}} \leq \pipes{\curlyb{b_1/ mB, \dotsc, b_{\ell}/ mB}} \leq \pipes{B/ mB} \leq m^d,
\]
proving the claim.
\qed\qedhere \end{proof}
\renewcommand{\qedsymbol}{}\end{proof}

\subsection{Weakly-minimal and 1-based theories} \label{subs:weak minimality}

\begin{defn}
A (complete\footnote{In this paper, when we say ``theory'', we mean a complete one.}) theory $T$ is called \emph{weakly-minimal} if $T$ is stable and has $U$-rank $1$. A structure is \emph{weakly minimal} if its complete theory is. 
\end{defn}

The following fact is a characterization of weakly-minimal theories in terms of definable sets. It must be well-known to experts, but we did not find it stated quite like this. The characterization follows from and is similar to \cite[Theorem 21]{WagnerQOM}, where it is proved that $T$ is weakly-minimal iff in a monster model of $T$, every formula is equivalent to a Boolean combination of $M$-definable sets and finite sets.
\begin{fact}
\label{weakly_minimal_equiv_conditions}
Let $T$ be a theory with monster model $\cM$. The following are equivalent: 
\begin{enumerate}
\item $T$ is weakly-minimal.
\item \label{weakly_minimal_equiv_conditions_intersection_every_model} For every model $M \models T$ and every non-algebraic formula $\phi \left( x,b \right)$ with $\left| x \right| = 1$ and $b \in \cM$,  we have $\phi \left( M,b \right) \neq \emptyset$. 
\item \label{weakly_minimal_equiv_conditions_intersection_some_model} There is a small model $M \models T$ such that, for every non-algebraic formula $\phi \left( x,b \right)$ with $\left| x \right| = 1$ and $b \in \cM$,  we have $\phi \left( M,b \right) \neq \emptyset$. 
\item For every model $M \models T$ and every $a,b \in M$ and $C \subseteq M$, if $a,b \notin \acl^{\eq}\left( C \right)$ and $\tp \left( a /  \acl^{\eq}\left( \emptyset \right)  \right) =
\tp \left( b /  \acl^{\eq}\left( \emptyset \right)  \right)$ 
then
$\tp \left( a /  \acl^{\eq}\left( C \right)  \right) =
\tp \left( b /  \acl^{\eq}\left( C \right)  \right)$. 
\end{enumerate}
\end{fact}

\begin{proof}
We show (1) $\implies$ (2) $\implies$ (3) $\implies$ (1) and (1) $\implies$ (4) $\implies$ (3). 

For (1) $\implies$ (2): Since $U \left( T \right ) = 1$ and $\phi \left( x,b \right)$ is non-algebraic, $\phi \left( x,b \right)$ does not fork over $0$, and so does not fork over $M$. Thus, we can extend $\phi \left( x,b \right)$ to a global type $p$ that does not fork over $M$. By stability, $p$ is finitely satisfiable in $M$, so $\phi \left( x,b \right)$ is realized in $M$.

(2) $\implies$ (3) is trivial. 

For (3) $\implies$ (1): To show that $U \left( T \right ) = 1$ it is enough to show that every formula $\phi \left( x,b \right)$ with $\left| x \right| = 1$ and $b \in \cM$ which divides over $\emptyset$, is algebraic. Therefore, suppose $\phi \left( x,b \right)$ is not algebraic. Let $\left( b_i \right)_{i < \omega}$ be a sequence indiscernible over $\emptyset$ such that $b_0 = b$ and such that $ \set{\phi \left( x,b_i \right)}{i < \omega} $ is inconsistent. Let $M$ be given by (3), and let  $\left( c_i \right)_{i < \omega}$ be a sequence indiscernible over $M$ with the same EM-type over $\emptyset$ as $\left( b_i \right)_{i < \omega}$. In particular, $ \set{\phi \left( x,c_i \right)}{i < \omega} $ is inconsistent, and $\phi \left( x,c_0 \right)$ is not algebraic. By (3), there is $a \in M$ realizing $\phi \left( x,c_0 \right)$. By indiscernibility over $M$, $a$ realizes $ \set{\phi \left( x,c_i \right)}{i < \omega} $, a contradiction.

To demonstrate stability, let $M$ be given by (3) and denote $\lambda := \left| M \right|$, $\kappa := 2^{ 2^{ \lambda } }$. Here, we count $1$-types over sets of size $\le \kappa$. Let $A \subseteq \cM$ such that $\left| A \right| \le \kappa$. Let $N \prec \cM$ be of size at most $\kappa$ such that $A \subseteq N$ and $M \subseteq N$ (so $M \prec N$). By (3), every non-algebraic type over $N$ is finitely satisfiable in $M$. Thus, the number of non-algebraic types over $N$ is bounded by the number of types over $N$ that are coheirs over $M$, which is bounded by $2^{ 2^{ \lambda } } = \kappa$. The number of algebraic types over $N$ is $\left| N \right| \le \kappa$. Together we get $ \left| S_1 \left( A \right) \right| \le \left| S_1 \left( N \right) \right| \le \kappa$. Therefore $T$ is stable.

For (1) $\implies$ (4): Let $M,C,a,b$ be as in (4). Let $p := \tp \left( a /  \acl^{\eq}\left( \emptyset \right)  \right) =
\tp \left( b /  \acl^{\eq}\left( \emptyset \right)  \right)$.
By stability, $p$ has a unique extension to a type $q$ over $\acl^{\eq}\left( C \right)$ which does not fork over $\acl^{\eq}\left( \emptyset \right)$. Since $a \notin \acl^{\eq}\left( C \right)$, and since $U \left( T \right ) = 1$, we get that $\tp \left( a /  \acl^{\eq}\left( C \right) \right)$ does not fork over $\emptyset$ and hence does not fork over $\acl^{\eq}\left( \emptyset \right)$. By the uniqueness of $q$, $\tp \left( a /  \acl^{\eq}\left( C \right) \right) = q$. Identically, $\tp \left( b /  \acl^{\eq}\left( C \right) \right) = q$, thus these types are equal.

For (4) $\implies$ (3): Let $M$ be a small $\left| L \right|^+$-saturated model. We claim that $M$ is as required. Let $\phi \left( x,b \right)$ be as in (3). Note that $\left| \acl^{\eq}\left( b \right) \right| \le \left| L \right|$, so, since $\phi \left( x,b \right)$ is non-algebraic, there exists $a \in \cM$ such that $\cM \models \phi \left( a,b \right)$ and $a \notin \acl^{\eq} \left( b \right)$. Let
\[
Q := \tp \left( a /  \acl^{\eq}\left( \emptyset \right)  \right) 
\cup
\set{x \neq d}{d \in \acl^{\eq} \left( b \right) }
\]
Thus, $Q$ is a partial type over $\acl^{\eq} \left( b \right)$, which is of size at most $\left| L \right|$, hence it has a realization $a' \in M$. So $a' \notin \acl^{\eq} \left( b \right)$ and
$\tp \left( a' /  \acl^{\eq}\left( \emptyset \right)  \right) =
\tp \left( a /  \acl^{\eq}\left( \emptyset \right)  \right)$. By (4), we obtain 
$\tp \left( a' /  \acl^{\eq}\left( b \right)  \right) =
\tp \left( a /  \acl^{\eq}\left( b \right)  \right)$, and in particular, $\cM \models \phi \left( a',b \right)$.
\end{proof}

\begin{defn}
\begin{enumerate}
\item A subset $X \subseteq M^k$ of a structure $M$ is called \emph{almost $0$-definable} if it is definable over $\acl^{\eq}\left( \emptyset \right)$.
\item Two structures $M$ and $N$ with the same universe are called \emph{almost $0$-interdefinable} if $N$ is an $\acl_M^{\eq}\left( \emptyset \right)$-reduct of $M$ and $M$ is an $\acl_N^{\eq}\left( \emptyset \right)$-reduct of $N$
\end{enumerate}
\end{defn}

\begin{defn}
A structure $M$ $0$-expanding an abelian group $\left( M,+ \right)$ is called a \emph{quasi-abelian structure} if, for every $k$, every definable subset $X \subseteq M^k$ is a Boolean combination of cosets of almost $0$-definable subgroups of $M^k$.
\end{defn}

\begin{rem}
\label{pure_abelian_group_is_quasi_abelian}
Every abelian group in the language of groups is quasi-abelian. This follows from quantifier elimination up to pp formulas.
\end{rem}

\begin{fact}[{\cite[Proposition 2.1]{HL2010}}]
\label{fact_when_is_a_quasi_abelian_structure_weakly_minimal}
Let $M$ be a quasi-abelian structure. $M$ is weakly-minimal if and only if every definable subgroup of $M$ is either finite or of finite index. This remains true if we replace ``definable'' with ``almost $0$-definable''.
\end{fact}

\begin{fact}[{\cite[Proposition 3.1]{ConantLaskowski2020}}]
\label{abelian_group_with_small_quotients_and_small_torsion_is_weakly_minimal}
An abelian group with small quotients and small torsion is weakly-minimal.
\end{fact}

Combining this with \cref{fact_torsion_free_abelian_group_of_finite_rank_has_small_quotients} we get:

\begin{cor}
\label{cor_torsion_free_abelian_group_of_finite_rank_is_weakly_minimal}
If $M = \roundedb{G,+}$ is a torsion-free abelian group of finite rank, then $M$ is weakly-minimal.
\end{cor}

We denote by $\ind$ the ternary relation of forking independence.

\begin{defn}
A stable theory $T$ with monster model $\cM$ is called \emph{1-based} if for every $A,B \subseteq \cM^{\eq}$ we have 
$A \underset{\acl^{\eq} \left( A \right) \cap \acl^{\eq} \left( B \right)}{\ind} B$.
A stable structure is called \emph{1-based} if its theory is 1-based.
\end{defn}

\begin{rem}
In the above definition, it is enough to consider only finite sets $A,B \subseteq \cM^{\eq}$. 
\end{rem}

\begin{fact}[{\cite[Theorem 4.1]{HP1987}}]
\label{HP_one_based_implies_quasi_abelian}
Let $M$ be a stable structure $0$-expanding an abelian group $\left( M,+ \right)$. Then $M$ is 1-based if and only if it is quasi-abelian.
\end{fact}

\begin{rem}
\label{pure_abelian_group_is_1_based}
In particular, by \cref{pure_abelian_group_is_quasi_abelian}, every abelian group in the language of groups is 1-based.
\end{rem}

\begin{fact}[{\cite[Theorem 2.1]{Loveys1990}}]
\label{Loveys_on_weakly_minimal_one_based}
Let $M$ be a weakly-minimal and 1-based structure $0$-expanding an abelian group $\left( M,+ \right)$ of unbounded exponent. Let $R$ be the ring of all endomorphisms of $\left( M,+ \right)$ that are definable in $M$, and let $M' := \left( M,+, f : f \in R \right)$. Then $M$ and $M'$ are interdefinable.
\end{fact}

\begin{rem}
\begin{enumerate}
\item In \cite{Loveys1990} $M$ is assumed to be an ``abelian structure'' rather than ``1-based'', but it is the same (see \cite[below Theorem 1.3]{Loveys1990}). 

\item The formulation of \cref{Loveys_on_weakly_minimal_one_based} is weaker than that in \cite{Loveys1990} and immediately follows from it. 
\end{enumerate}
\end{rem}

\begin{defn}\label{def:acl-independence}
    When $\acl$ satisfies exchange (and therefore defines a pregeometry), we denote by $\ind^{a}$ the ternary relation of $\acl$-independence: $A \ind^{a}_C B$ if for every finite tuple $a$ from $A$ we have $\dim \left( a / CB \right) = \dim \left( a / C \right)$.
\end{defn}

The following is folklore. 
\begin{fact}
\label{u_rank_1_acl_facts}
Let $T$ be a (supersimple) theory of $SU$-rank $1$ (so in particular this applies when $T$ is weakly-minimal), with monster model $\cM$. Then:
\begin{enumerate}
\item $\acl$ satisfies exchange. 
\item For every (real) tuple $a \in \cM$ and $C \subseteq \cM$ we have $SU \left( a / C \right) = \dim \left( a / C \right)$, where $\dim \left( a / C \right)$ is the $\acl$-dimension of $a$ over $C$.
\item 
\label{u_rank_1_forking_indep_coincides_with_acl_indep}
For every (real) $A,B,C \subseteq \cM$ we have $A \ind_C B \iff A \ind^{a}_C B$, where $\ind^{a}$ is defined above.
\end{enumerate}
\end{fact}

\subsection{Weak and geometric elimination of imaginaries} \label{subs:weak and geometric EI}

\begin{defn}
\begin{enumerate}
\item A structure $M$ has \emph{weak elimination of imaginaries} (WEI) if for every $e \in M^{\eq}$ there is a tuple $a \in M$ such that $e \in \dcl^{\eq} \left( a \right)$ and $a \in \acl^{\eq} \left( e \right)$.
\item A structure $M$ has \emph{geometric elimination of imaginaries} (GEI) if for every $e \in M^{\eq}$ there is a tuple $a \in M$ such that $e \in \acl^{\eq} \left( a \right)$ and $a \in \acl^{\eq} \left( e \right)$.
\item A theory $T$ has WEI (respectively, GEI) if every $M \models T$ has WEI (respectively, GEI).
\end{enumerate}
\end{defn}

\begin{rem}
Clearly, WEI implies GEI. Also, to check that a theory $T$ has WEI (respectively, GEI), it is sufficient to check it for a single $\left| T \right|^+$-saturated model of $T$.
\end{rem}

\begin{fact}
\label{weakly_minimal_and_acl0_is_a_model_implies_wei}
Let $T$ be a weakly-minimal theory, such that $N := \acl \left( \emptyset \right)$ is a model of $T$. Then $T$ has WEI.
\end{fact}

\begin{proof}
This is a variation of \cite[Lemma 1.6]{Pillay1998}. Let $M \models T$, and let $e \in M^{\eq}$. Thus, there exist a $0$-definable function $f$ and a tuple $b = \left( b_1, \dots, b_n \right) \in M^n$ such that $f \left( b \right) = e$. Let $A := \acl^{\eq} \left( e \right) \cap M$. So $N \subseteq A$.

We will find a tuple $c = \left( c_1, \dots, c_n \right) \in A^n$ such that $f \left( c \right) = e$. We find $c_i$ by recursion on $i$. Suppose we already found $c_1, \dots, c_{i-1} \in A$ such that for some $d_i, \dots, d_n \in M$ we have $e = f \left( c_1, \dots, c_{i-1}, d_i, \dots, d_n \right)$ (for $i=1$ this holds with $\left( d_1, \dots, d_n \right) = \left( b_1, \dots, b_n \right)$). Let $\phi_i \left( x_i \right)$ be the following formula 
\[
\exists x_{i+1}, \dots, x_n \left( e = f(c_1, \dots, c_{i-1}, x_i, x_{i+1}, \dots, x_n) \right).
\]

So $d_i \in \phi_i \left( M \right) \neq \emptyset$. We have two cases: 
(1): $\phi_i \left( M \right)$ is finite. In this case, $d_i \in \acl^{\eq} \left( A \right) \cap M = A$. We let $c_i := d_i$.
(2): $\phi_i \left( M \right)$ is infinite. In this case, by weak minimality (\cref{weakly_minimal_equiv_conditions} (\ref*{weakly_minimal_equiv_conditions_intersection_every_model})) we get $\phi_i \left( N \right) \neq \emptyset$. Let $c_i$ be any element in $\phi_i \left( N \right)$. So $c_i \in N \subseteq A$.

So we found $c \in A^n$ such that $f \left( c \right) = e$, hence  
$e \in \dcl^{\eq} \left( c \right)$ and $c \in \acl^{\eq} \left( e \right)$.
\end{proof}

\begin{prop}
\label{equivalent_cond_for_1_based_under_gei}
Let $T$ be a stable theory with monster model $\cM$. Suppose that $T$ has GEI. Then $T$ is 1-based if and only if for every (real) tuples $a,b \in \cM$ we have 
$a \underset{\acl \left( a \right) \cap \acl \left( b \right)}{\ind} b$.
\end{prop}

\begin{proof}
Suppose $T$ is 1-based. Let $a,b \in \cM$ be real tuples. Then
$a \underset{\acl^{\eq} \left( a \right) \cap \acl^{\eq} \left( b \right)}{\ind} b$. We demonstrate that $\acl^{\eq} \left( a \right) \cap \acl^{\eq} \left( b \right) = \acl^{\eq} \left( \acl \left( a \right) \cap \acl \left( b \right) \right)$: In every theory, the right-hand side is contained in the left-hand side. Let $e \in \acl^{\eq} \left( a \right) \cap \acl^{\eq} \left( b \right)$. By GEI, there exists a real tuple $d \in \cM$ such that $\acl^{\eq} \left( e \right) = \acl^{\eq} \left( d \right)$. So also $d \in \acl^{\eq} \left(a \right) \cap \acl^{\eq} \left(b \right)$, and since $d$ is real, $d \in \acl \left( a \right) \cap \acl \left( b \right)$. Therefore $e \in \acl^{\eq} \left( \acl \left( a \right) \cap \acl \left( b \right) \right)$. So $\acl^{\eq} \left( a \right) \cap \acl^{\eq} \left( b \right) = \acl^{\eq} \left( \acl \left( a \right) \cap \acl \left( b \right) \right)$, hence $a \underset{\acl^{\eq} \left( \acl \left( a \right) \cap \acl \left( b \right) \right)}{\ind} b$, and therefore $a \underset{\acl \left( a \right) \cap \acl \left( b \right)}{\ind} b$.

For the other direction, let $a,b \in \cM^{\eq}$. By GEI, there exist real tuples $a', b' \in \cM$ such that $\acl^{\eq} \left( a' \right) = \acl^{\eq} \left( a \right)$ and $\acl^{\eq} \left( b' \right) = \acl^{\eq} \left( b \right)$. By the assumption, $a' \underset{\acl \left( a' \right) \cap \acl \left( b' \right)}{\ind} b'$, so $\acl^{\eq} \left( a' \right) \underset{\acl \left( a' \right) \cap \acl \left( b' \right)}{\ind} \acl^{\eq} \left( b' \right)$, and hence $\acl^{\eq} \left( a' \right) \underset{\acl^{\eq} \left( a' \right) \cap \acl^{\eq} \left( b' \right)}{\ind} \acl^{\eq} \left( b' \right)$. Therefore $\acl^{\eq} \left( a \right) \underset{\acl^{\eq} \left( a \right) \cap \acl^{\eq} \left( b \right)}{\ind} \acl^{\eq} \left( b \right)$, and hence $a \underset{\acl^{\eq} \left( a \right) \cap \acl^{\eq} \left( b \right)}{\ind} b$.
\end{proof}

\subsection{Archimedean ordered abelian groups} \label{subs:archimedean OAGs}

\begin{defn}
An \emph{ordered abelian group} is an abelian group $\left( G,+ \right)$ together with a linear order $<$ on $G$ such that for all $a,b,c \in G$, $a < b \implies a+c < b+c$. 
An ordered abelian group is called \emph{discrete} if it has a minimal positive element, and \emph{dense} otherwise.
An ordered abelian group is called \emph{archimedean} if for every $a,b \in G$ there exists $n \in \Z$ such that $b < na$.
\end{defn}

\begin{rem}
\begin{enumerate}
\item In a dense ordered abelian group, the order is indeed dense.
\item Archimedean ordered abelian groups do not have non-trivial
convex subgroups.
\item Every discrete archimedean ordered abelian group is isomorphic to $\left( \Z,+,< \right)$.
\end{enumerate}
\end{rem}

\begin{fact}
\label{QE_archimedean_ordered_abelian_group}
Let $\left( G,+,< \right)$ be an archimedean ordered abelian group.
\begin{enumerate}
\item If $G$ is dense,  $\Th{G}$ has quantifier elimination in the language $\curlyb{+,-,<,0} \cup \set{\equiv_m}{2 \le m < \omega}$.
\item If $G$ is discrete,  $\Th{G}$ has quantifier elimination in the language $\curlyb{+,-,<,0,1} \cup \set{\equiv_m}{2 \le m < \omega}$.
\end{enumerate}
Here, $a \equiv_m b$ is interpreted as $a-b \in mG$ and $1$ is interpreted as the minimal positive element.
\end{fact}

\begin{proof}
This is a special case of \cite[Proposition 3.4]{HaleviHasson2019}. In our case, since $G$ is archimedean, it has no non-trivial
convex subgroups. In particular, this implies that all spines are trivial.

Note that ``$-$'' is not required for quantifier elimination: however, we added it for convenience.
\end{proof}

\begin{defn}
Let $\left( G,+,< \right)$ be an ordered abelian group. A \emph{generalized interval with rational endpoints} is a nonempty set of the form $\set{g \in G}{ng \ge a} \cap \set{g \in G}{mg \le b}$, where $1 \le n,m \in \omega$, $a,b \in G \cup \curlyb{\pm \infty}$, and either or both inequalities may be replaced by strict inequalities. As usual, a generalized interval with neither minimum nor maximum is called \emph{open}.
\end{defn}

We will denote a set of the form $\set{g \in G}{ng > a} \cap \set{g \in G}{mg < b}$ by $\roundedb{\frac{a}{n} , \frac{b}{m}}$, and similarly for other kinds of generalized intervals with rational endpoints. This notation makes sense, because an ordered abelian group is torsion-free and the order extends uniquely to the divisible hull.
Thus, such an interval in $G$ is precisely the intersection with $G$ of the corresponding interval in the divisible hull of $G$.

For $1 \le n < \omega$ we will use the convention that $n \cdot \infty = \frac{\infty}{n} = \infty$ and $n \cdot \roundedb{-\infty} = \frac{-\infty}{n} = -\infty$.

\begin{defn}
Let $\left( G,+,< \right)$ be an ordered abelian group. For a subset $A \subseteq G$, we denote by $\spn{A}$ the linear span of $A$ over $\Q$, computed in the divisible hull of $G$ (so it is a subset of the divisible hull). We also denote $\spnplus{A} := \spn{A \cup \dcl \roundedb{0}}$. 
\end{defn}

If $\left( G,+,< \right)$ is elementarily equivalent to an archimedean ordered abelian group, there are two cases: If $G$ is dense, then by \cref{QE_archimedean_ordered_abelian_group} we have $\dcl \roundedb{0} = \curlyb{0}$ and $\spnplus{A} = \spn{A}$. Otherwise, $G$ is discrete, so elementarily equivalent to $\roundedb{\Z,+,<}$. By \cref{QE_archimedean_ordered_abelian_group} we may assume that $G$ is an elementary extension of $\roundedb{\Z,+,<}$, so $\dcl \roundedb{0} = \Z$ and $\spnplus{A} = \spn{A \cup \curlyb{1}}$.

\begin{cor}
\label{cor_shape_of_unary_subsets_archimedean_small_quotients}
Let $\left( G,+,< \right)$ be elementarily equivalent to an archimedean ordered abelian group with small quotients. Then, for every unary definable subset $D \subseteq G$ there exists some $1 \leq m <\omega$ and $D_i$ for $i=1,\ldots, N$ such that $D=\bigcup_{i=1}^N D_i$ and for each $i$, either
\begin{itemize}

\item $D_i = \curlyb{c_i}$ is a singleton, or

\item $D_i = I_i \cap \roundedb{mG + g_i}$, where $g_i \in G$ and $I_i$ is an open generalized interval with rational endpoints.

\end{itemize}
Note that $m$ does not depend on $i$. 

In addition, suppose $D$ is definable over $\Bar{b}$. Then, for each $i$, if $D_i = \curlyb{c_i}$ is a singleton then $c_i \in \spnplus{\Bar{b}} \cap G$, and otherwise the endpoints of $I_i$ are in $\spnplus{\Bar{b}} \cup \curlyb{\pm \infty}$. Note that we do not require that the $g_i$'s are from $\Bar{b}$.
\end{cor}

\begin{proof}
By \cref{QE_archimedean_ordered_abelian_group}, $D$ is definable by a Boolean combination of atomic formulas over $\Bar{b}$ in either $\curlyb{+,-,<,0} \cup \set{\equiv_m}{2 \le m < \omega}$ or $\curlyb{+,-,<,0,1} \cup \set{\equiv_m}{2 \le m < \omega}$, depending on whether $G$ is dense or discrete. We note the following points:

\begin{itemize}

\item A term $t \roundedb{\Bar{x}}$ (without parameters) is a linear combination over $\Z$ of either the elements of $\Bar{x}$, if $G$ is dense, or the elements of $\Bar{x}$ and $1$, if $G$ is discrete.

\item A formula of the form $t_1 \roundedb{x , \Bar{b}} = t_2 \roundedb{x , \Bar{b}}$ is equivalent to a formula of the form $nx = t \roundedb{\Bar{b}}$ with $n < \omega$. If $n=0$, then this formula is equivalent to either $\top$ or $\bot$. In addition, if $n \ge 1$, then if $\frac{t \roundedb{\Bar{b}}}{n} \in G$ then this formula defines $\frac{t \roundedb{\Bar{b}}}{n}$, otherwise this formula is equivalent to $\bot$. This also implies that a conjunction of formulas such that one of them is $nx = t \roundedb{\Bar{b}}$ with $n \ge 1$, is equivalent to  either $nx = t \roundedb{\Bar{b}}$ or to $\bot$.

\item A formula of the form $\neg \roundedb{ t_1 \roundedb{x , \Bar{b}} = t_2 \roundedb{x , \Bar{b}} }$ is equivalent to $\roundedb{ t_1 \roundedb{x , \Bar{b}} < t_2 \roundedb{x , \Bar{b}} } \vee t_2 \roundedb{x , \Bar{b}} < \roundedb{ t_1 \roundedb{x , \Bar{b}} }$.

\item A formula of the form $t_1 \roundedb{x , \Bar{b}} < t_2 \roundedb{x , \Bar{b}}$ is equivalent to a formula of either the form $nx < t \roundedb{\Bar{b}}$ or the form $ t \roundedb{\Bar{b}} < nx$, with $n < \omega$. If $n=0$, then this formula is equivalent to either $\top$ or $\bot$.

\item A formula of the form $\neg \roundedb{ t_1 \roundedb{x , \Bar{b}} < t_2 \roundedb{x , \Bar{b}    } }$ is equivalent to $\roundedb{  t_2 \roundedb{x , \Bar{b}} < t_1 \roundedb{x , \Bar{b}} } \vee \roundedb{ t_1 \roundedb{x , \Bar{b}} = t_2 \roundedb{x , \Bar{b}} }$.

\item A formula of the form $\roundedb{n_1x < t_1 \roundedb{\Bar{b}}} \wedge \roundedb{n_2x < t_2 \roundedb{\Bar{b}}}$ with $n_1, n_2 \ge 1$ is equivalent to a formula of the form $nx < t \roundedb{\Bar{b}}$ with $n \ge 1$, and analogously for $>$ in place of $<$ (i.e., switching the order in the inequality). 

\item Let $2 \le m < \omega$. Since $G$ has small quotients, we can take a finite set $g_1, \dots, g_N \in G$ of representatives for the cosets of $mG$. A formula of either the form $t_1 \roundedb{x , \Bar{b}} \equiv_m t_2 \roundedb{x , \Bar{b}}$ or the form $\neg \roundedb{ t_1 \roundedb{x , \Bar{b}} \equiv_m t_2 \roundedb{x , \Bar{b}} }$ is equivalent to $\bigvee_{i \in F} x \equiv_m g_i$ for some $F \subseteq \curlyb{1, \dots, N}$. 

\item Similarly, if $2 \le m_1,m_2 < \omega$, $m_1 | m_2$ and $g_1, \dots, g_N \in G$ are representatives for the cosets of $m_2G$, then a formula of the form $x \equiv_{m_1} h$ is equivalent to $\bigvee_{i \in F} x \equiv_{m_2} g_i$ for some $F \subseteq \curlyb{1, \dots, N}$. 

\item A formula of the form $\roundedb{x \equiv_{m} g} \wedge \roundedb{x \equiv_{m} h}$ is equivalent to either $x \equiv_{m} g$, if $g \equiv_{m} h$, or to $\bot$, if $g \not\equiv_{m} h$.

\end{itemize}

The conclusion follows easily from these points.
\end{proof}

\begin{lem}
\label{lem_intersection_of_interval_with_coset}
Let $\left( G,+,< \right)$ be elementarily equivalent to an archimedean ordered abelian group with small quotients. Then, for every $1 \le m < \omega$, $g \in G$, $1 \le n,k < \omega$, and $a,b \in G \cup \curlyb{\pm \infty}$, if $\roundedb{ \frac{a}{n} , \frac{b}{k}}$ is infinite then $\roundedb{ \frac{a}{n} , \frac{b}{k}} \cap \roundedb{mG + g}$ is infinite. 
\end{lem}

\begin{proof}
First, we demonstrate that it is sufficient to prove this for $n=k=1$. Since $\roundedb{ \frac{a}{n} , \frac{b}{k}} = \roundedb{ \frac{ka}{nk} , \frac{nb}{nk}}$, we may assume that $n=k$. Since $\roundedb{ \frac{a}{n} , \frac{b}{n}}$ is infinite and $G$ is torsion-free, also $\roundedb{a,b}$ is infinite. Applying the lemma with $nm$ in place of $m$, $ng$ in place of $g$, $n=k=1$, and $a,b$, we find that 
$\roundedb{a,b} \cap \roundedb{nmG + ng}$ is infinite. Every element in this set is of the form $nmh + ng = n \roundedb{mh + g}$ for some $h \in G$, so $mh + g \in \roundedb{ \frac{a}{n} , \frac{b}{n}} \cap \roundedb{mG + g}$. And if $n \roundedb{mh_1 + g} \neq n \roundedb{mh_2 + g}$ then $\roundedb{mh_1 + g} \neq \roundedb{mh_2 + g}$. Therefore $\roundedb{ \frac{a}{n} , \frac{b}{n}} \cap \roundedb{mG + g}$ is infinite.

Second, we show that it is enough to prove this for $a,b \in G$. Indeed, suppose that $b = \infty$. By replacing $g$ with $\roundedb{1-m}g$ we may assume that $g \ge 0$. If $a>0$, let $c := 2a > a$; otherwise, let $c>0$ be any positive element. Now for every $1 \le k < \omega$, $a < mkc + g \in mG + g$, so $\roundedb{a , \infty} \cap \roundedb{mG + g}$ is infinite. The case when $a = -\infty$ is dealt with analogously.

If $G$ is discrete, then it is elementarily equivalent to $\left( \Z,+,< \right)$, and (using quantifier elimination) we may assume that $G$ is an elementary extension of $\left( \Z,+,< \right)$. In this case, the set $\set{a + n}{1 \le n < \omega}$ is contained in $\roundedb{a,b}$ and has an infinite intersection with $mG + g$.

So suppose $G$ is dense. Then $\roundedb{a,b}$ is infinite if and only if $a < b$, and, by iterating, $\roundedb{a,b} \cap \roundedb{mG + g}$ is infinite if and only if it is nonempty. Thus, for each $m$, the statement we want to prove is equivalent to ``for every $a,b,g \in G$, if $a < b$ then $\roundedb{a,b} \cap \roundedb{mG + g}$ is nonempty'', which is a first-order sentence. Hence, we may assume that $G$ itself is archimedean.
Also, note that it is enough to prove this statement for $a=0$. 

Thus, let $1 \le m < \omega$ and let $b,g \in G$ such that $b > 0$. We show that $\roundedb{0,b} \cap \roundedb{mG + g}$ is nonempty.
Since $G$ has small quotients, $G/mG$ is finite. Since $\roundedb{0,b}$ is infinite, there are $0 < d_1 < d_2 < b$ such that $c := d_2 - d_1 \in mG$.
Since $G$ is archimedean, there exists $n \in \Z$ such that $nc < g \le \roundedb{n+1}c$. Let $e := g - nc$. Then $0 < e \le c < b$ and $e \in mG + g$, as required.
\end{proof}

\begin{cor}
\label{cor_acl_archimedean_small_quotients}
Let $\roundedb{G,+,<}$ be elementarily equivalent to an archimedean ordered abelian group with small quotients. If $G$ is discrete, let $N := \roundedb{G,+,<,1}$ and $M := \roundedb{G,+,1}$, otherwise let $N := \roundedb{G,+,<}$ and $M := \roundedb{G,+}$. Then for every $A \subseteq G$, $\acl_{N}\roundedb{A} = \spnplus{A} \cap G = \acl_{M}\roundedb{A}$.
\end{cor}

\begin{proof}
Clearly $\spnplus{A} \cap G  \subseteq  \dcl_{M}\roundedb{A} \subseteq  \acl_{M}\roundedb{A} \subseteq \acl_{N}\roundedb{A}$. We show $\acl_{N}\roundedb{A} \subseteq \spnplus{A} \cap G$.
Let $c \in \acl_{N}\roundedb{A}$. Thus, there exists a finite unary subset $D \subseteq G$ definable over $A$ such that $c \in D$. By \cref{cor_shape_of_unary_subsets_archimedean_small_quotients}, $D = \bigcup_{i=1}^N D_i$ such that, for each $i$, either
\begin{enumerate}

\item $D_i = \curlyb{c_i}$ is a singleton with $c_i \in \spnplus{A}$, or

\item $D_i = I_i \cap \roundedb{mG + g_i}$, where $g_i \in G$ and $I_i$ is an open generalized interval with rational endpoints, such that the endpoints of $I_i$ are in $\spnplus{A} \cup \curlyb{\pm \infty}$,

\end{enumerate}
and $1 \le m < \omega$ does not depend on $i$. 
Thus, $c \in D_i$ for some $i$. If $D_i$ is as in the first case, then $c = c_i \in \spnplus{A}$ as required. 
So suppose that $D_i$ is as in the second case. Write $I_i = \roundedb{a,b}$ with $a,b \in \spnplus{A} \cup \curlyb{\pm \infty}$. Since $D$ is finite, also $D_i = I_i \cap \roundedb{mG + g_i}$ is finite. By \cref{lem_intersection_of_interval_with_coset}, $I_i$ is finite, which is possible only when $G$ is discrete. Since $I_i$ is finite, $a,b \notin \curlyb{\pm \infty}$, so $a,b \in \spnplus{A}$. In particular, this means that $I_i$ is bounded from below; therefore, since $G$ is discrete, $I_i$ has a minimum $a'$. So $a' \in I_i$ but $a' - 1 \notin I_i$, hence $a < a' \le a + 1$. Since $a \in \spnplus{A}$, there exists $1 \le n < \omega$ such that $na \in G$. Now $na < na' \le na + n$ and $na, na' \in G$, so there exists $1 \le k \le n$ such that $na' = na + k$, and hence $a' = a + \frac{k}{n} \in \spnplus{A}$. Since $c \in I_i$ and $I_i$ is finite, there exists $l < \omega$ such that $c = a' + l \in \spnplus{A}$, as required.
\end{proof}

\section{Preserving group homomorphisms} \label{sec:preservation of hom}

In this section, we will give sufficient conditions for when an expansion of a group does not add new definable group homomorphisms. Namely, we have a structure $M$ expanding a group $\left(G,\cdot \right)$, and some expansion $N$, and we ask whether $N$ adds new $N$-definable group homomorphisms $G \to G$.
We give two lemmas which ensure that in two important cases, the answer is no. The first case is when $G$ is an $R$-group (see \cref{def_R-group}), and the second is when $G$ is an abelian group with small quotients (see \cref{def_small_quotients_and_torsion}). In both cases we require that $\dcl_{N} \subseteq \acl_{M}$, namely, that for every set $A \subseteq M$, $\dcl_{N} \left( A \right) \subseteq \acl_{M} \left( A \right)$. We also require some saturation.

\begin{lem}
\label{lem_homomorphism_preservation_R_groups}
Suppose that $N$ is an $\omega$-saturated structure $0$-expanding a structure $M$ that expands an $R$-group $\left( G, \cdot \right)$. 
Let $f : G \to G$ be a group homomorphism definable in $N$. If there exists a finite set $F \subseteq G$ such that for every $g\in G$, $f\left(g\right) \in \acl_{M}\left(Fg\right)$, then $f$ is definable in $M$.

In particular, if  $\dcl_{N} \subseteq \acl_{M}$, then every group homomorphism $f : G \to G$ definable in $N$ is definable in $M$.
\end{lem}

\begin{proof}
The group operation is definable in $M$, possibly using finitely many parameters. By naming these parameters in $M$ and $N$, we may assume that $M$ $0$-expands $\left( G, \cdot \right)$. Since we named only finitely many elements, $N$ remains $\omega$-saturated.

Let $L$ be the language of $M$ and $L'$ be the language of $N$. We may assume that $L \subseteq L'$. By adding the finite set $F$ and the finitely many parameters in the definition of $f$ to both $M$ and $N$, we may assume that $f$ is $\emptyset$-definable in $N$ and that for every $g\in G$, $f\left(g\right)\in\acl_{M}\left(g\right)$.
Thus, for all $g$, there exists an $L$-formula $\varphi_{g}\left(x,y\right)$
(over $\emptyset$) such that $N\models\varphi_{g}\left(g,f\left(g\right)\right)$
and $\varphi_{g}\left(g,M\right)$ is finite. Hence, by saturation
of $N$, for all $p\in S_{L'}\left(\emptyset\right)$ there is an
$L$-formula $\varphi_{p}\left(x,y\right)$ such that $\varphi_{p}\left(x,f\left(x\right)\right)\in p$
and $\exists^{\leq n_{p}}y\varphi_{p}\left(x,y\right)\in p$ for some
$n_{p}<\omega$. By compactness, there are finitely many formulas corresponding
to finitely many types $p_{0},\ldots,p_{k}$ which cover all of $S_{L'}\left(\emptyset\right)$.
Taking the disjunction of these formulas, we obtain an $L$-formula $\varphi\left(x,y\right)$
(over $\emptyset$) such that, for all $g\in G$, $N\models\varphi\left(g,f\left(g\right)\right)$
(in other words, $\varphi$ contains the graph of $f$) and $\varphi\left(g,M\right)$
is uniformly bounded by some number. Given such a formula $\varphi\left(x,y\right)$
(maybe with parameters from $M$), let $m_{\varphi}:=\max\set{\left|\varphi\left(g,M\right)\right|}{g\in G}$.
Let $m:=\min\set{m_{\varphi}}{\varphi\text{ as above}}$, and let $\varphi_{0}$
be such that $m_{\varphi_{0}}=m$.

({*}) Note that if $\varphi'$ implies $\varphi_{0}$ and contains
the graph of $f$ then $m_{\varphi'}=m_{\varphi_{0}}=m$.

Let $\varphi_{1}\left(x,y\right):=\varphi_{0}\left(x,y\right)\land\forall x'\exists y'\left(\varphi_{0}\left(x',y'\right)\land\varphi_{0}\left(xx',yy'\right)\right)$.
Note that $\varphi_{1}$ also contains the graph of $f$: for all
$g\in G$, $\varphi_{0}\left(g,f\left(g\right)\right)$ holds and
for all $g'$, $f\left(gg'\right)=f\left(g\right)f\left(g'\right)$
so that $\varphi_{0}\left(gg',f\left(g\right)f\left(g'\right)\right)$
and $\varphi_{0}\left(g',f\left(g'\right)\right)$. By ({*}), for
some $g_{*}\in G$, $\left|\varphi_{1}\left(g_{*},M\right)\right|=m$.
Let $\varphi_{1}\left(g_{*},M\right)=\set{f\left(g_{*}\right)h_{i}}{i<m}$
where $h_{i}\in G$ are distinct for $i<m$ and $h_{0}=1$.

Let $K:=m^{2}+1$.

Let $\varphi_{2}\left(x,y\right):=\bigwedge_{k<K}\varphi_{0}\left(x^{k+1},y^{k+1}\right)\land\bigwedge_{k<K}\varphi_{0}\left(g_{*}x^{k+1},f\left(g_{*}\right)y^{k+1}\right)$.
Note that $\varphi_{2}$ implies $\varphi_{0}$ (by putting $k=0$)
and contains the graph of $f$ as well (because $f\left(g^{k}\right)=f\left(g\right)^{k}$
and $f\left(g_{*}g^{k}\right)=f\left(g_{*}\right)f\left(g\right)^{k}$).
Again by ({*}), $m_{\varphi_{2}}=m$. We claim that $\varphi_{2}$
defines $f$.

Let $g_{\dagger}\in G$ be such that $\left|\varphi_{2}\left(g_{\dagger},M\right)\right|=m$
and write $\varphi_{2}\left(g_{\dagger},M\right)=\set{f\left(g_{\dagger}\right)c_{i}}{i<m}$
where $c_{i}\in G$ are distinct for $i<m$ and $c_{0}=1$.

We now note the following points:
\begin{enumerate}
\item For all $k<K$, $\varphi_{0}\left(g_{\dagger}^{k+1},M\right)=\set{\left(f\left(g_{\dagger}\right)c_{i}\right)^{k+1}}{i<m}$.
Why? as $\varphi_{2}\left(g_{\dagger},f\left(g_{\dagger}\right)c_{i}\right)$ holds
we get $\supseteq$ by the definition of $\varphi_{2}$. On the other
hand as the group is an $R$-group, $\left(f\left(g_{\dagger}\right)c_{i}\right)^{k+1}\neq\left(f\left(g_{\dagger}\right)c_{j}\right)^{k+1}$
for $i\neq j$, and as $m_{\varphi_{0}}=m$, $\left|\varphi_{0}\left(g_{\dagger}^{k+1},M\right)\right|\leq m$
so we obtain equality.
\item For the same reason as (1), for all $k<K$, $\varphi_{0}\left(g_{*}g_{\dagger}^{k+1},M\right)=\set{f\left(g_{*}\right)\left(f\left(g_{\dagger}\right)c_{i}\right)^{k+1}}{i<m}$.
\item For $i<m$ and $k<K$, as $\varphi_{1}\left(g_{*},f\left(g_{*}\right)h_{i}\right)$
holds, applying the second half of $\varphi_{1}$ with $x'=g_{\dagger}^{k+1}$,
by (1) we get some $j=j_{i,k}<m$ such that $y'=\left(f\left(g_{\dagger}\right)c_{j}\right)^{k+1}$
works: $\varphi_{0}\left(g_{*}g_{\dagger}^{k+1},f\left(g_{*}\right)h_{i}\left(f\left(g_{\dagger}\right)c_{j}\right)^{k+1}\right)$
holds.
\item By (2) and (3) we find that for all $i<m$ and $k<K$ there is some
$j'=j'_{i,k}$ such that $f\left(g_{*}\right)\left(f\left(g_{\dagger}\right)c_{j'}\right)^{k+1}=f\left(g_{*}\right)h_{i}\left(f\left(g_{\dagger}\right)c_{j}\right)^{k+1}$
which implies 
\[
\left(f\left(g_{\dagger}\right)c_{j'}\right)^{k+1}=h_{i}\left(f\left(g_{\dagger}\right)c_{j}\right)^{k+1}.
\]
\end{enumerate}
Assume towards contradiction that $m>1$ and fix $i=1$. For all $k<K$
we found some $j_{1,k}$ and $j'_{1,k}$, both $<m$, and as $h_{1}\neq1$,
they must be distinct by (4). As $K=m^{2}+1$, by pigeonhole there
must be $k<k'$ such that $\left(j_{1,k},j'_{1,k}\right)=\left(j_{1,k'},j'_{1,k'}\right)$.
Denote this pair by $\left(j,j'\right)$. Let $a:=f\left(g_{\dagger}\right)c_{j'}$,
$b:=f\left(g_{\dagger}\right)c_{j}$. So we have $a^{k+1}=h_{1}b^{k+1}$
and $a^{k'+1}=h_{1}b^{k'+1}$, hence substituting $h_{1}=a^{k+1}\left(b^{k+1}\right)^{-1}$
in the second equation we get $a^{k'-k}=b^{k'-k}$. As $G$ is an
$R$-group, $a=b$, so $j=j'$ --- contradiction.
\end{proof}


\begin{lem}
\label{lem_homomorphism_preservation_small_quotients} 
Suppose that $N$ is an $\omega$-saturated structure $0$-expanding a structure $M$ that expands an abelian group $\left( G, + \right)$ with small quotients.
Let $f : G \to G$ be a group homomorphism definable in $N$. If there exists a finite set $F \subseteq G$ such that for every $g\in G$, $f\left(g\right) \in \acl_{M}\left(Fg\right)$, then $f$ is definable in $M$.

In particular, if  $\dcl_{N} \subseteq \acl_{M}$, then every group homomorphism $f : G \to G$ definable in $N$ is definable in $M$.
\end{lem}

\begin{proof}
The proof is an elaboration of the idea of the proof of \cref{lem_homomorphism_preservation_R_groups}.
Exactly as in there, let $L$ be the language of $M$ and $L'$ the
language of $N$, we may assume that $L \subseteq L'$, 
that $M$ $0$-expands $\left( G, + \right)$, that $f$ is $\emptyset$-definable in $N$, and that for every $g\in G$, $f\left(g\right)\in\acl_{M}\left(g\right)$.

By compactness, we obtain an $L$-formula $\varphi\left(x,y\right)$
(over $M$) containing the graph of $f$ and such that for $\varphi\left(g,M\right)$
is uniformly bounded. Again, we let $m_{\varphi}:=\max\set{\left|\varphi\left(g,M\right)\right|}{g\in G}$
and $m:=\min\set{m_{\varphi}}{\varphi\text{ as above}}$. Let $\varphi_{0}$
be such that $m_{\varphi_{0}}=m$. Again, we let $K:=m^{2}+1$.

Given an $L$-formula $\psi\left(x,y\right)$ (over $M$) and a group
$H\leq G$, say that $\psi$ \emph{contains (the graph of) $f$ up
to $H$ }if $\psi\left(G\right)+\left(\left\{ 0\right\} \times H\right)$
contains $f$, i.e., whenever $f\left(g\right)=g'$, for some $e\in H$,
$M\models\psi\left(g,g'+e\right)$. For a set $C\subseteq G$, let
$C/H := \set{c+H}{c \in C}.$ Let $m_{\psi,H} := \sup\set{\left|\psi\left(g,M\right)/H\right|}{g\in G}$.
For $0<t<\omega$, let $m_{\psi,t}:=m_{\psi,H_{t}}$ where $H_{t}:=\set {g}{tg=0}$.
In this notation, for $\varphi$ as in the previous paragraph, $m_{\varphi}=m_{\varphi,1}$.
 Thus:
\[
m_{*}:=\min\set{m_{\psi,t}}{1\leq t<\omega,\psi\text{ contains }f\text{ up to }H_{t}}\leq m.
\]

If $m_{*}=1$, then we are done: let $\psi,t$ be such that $1=m_{\psi,t}$.
It follows that $\psi\left(g,f\left(g\right)+e\right)$ for some $e\in H_{t}$
and that if $\psi\left(g,h_{1}\right),\psi\left(g,h_{2}\right)$ hold
then $t\left(h_{1}-h_{2}\right)=0$. Let $\psi'\left(x,y\right):=\exists x'y'\left(x=tx'\land y=ty'\land\psi\left(x',y'\right)\right)$.
Then, for all $g\in tG$, $\psi'\left(g,f\left(g\right)\right)$ holds,
as witnessed by $g',h'$ where $tg'=g$ and $h'=f\left(g'\right)+e$
for some $e\in H_{t}$. Conversely, for such $g$, if $\psi'\left(g,h\right)$
holds, then there are $g',h'$ such that $\psi\left(g',h'\right)$
and $tg'=g,th'=h$. As $m_{\psi,t}=1$, It follows that $h'-f\left(g'\right)\in H_{t}$
so that $h=tf\left(g'\right)=f\left(g\right)$. Together we got that
$\psi'$ defines $f$ on $tG$. But since $tG$ has finite index in
$G$, we can now define $f$ on all~$G$: for every coset $C=tG+h$,
$f\restriction C$ is defined by $f\left(g\right)=f\left(g-h\right)+f\left(h\right)$,
i.e., by $\psi'\left(x-h,y-f\left(h\right)\right)$.

({*}) Suppose that $\psi,t$ are any pair such that $m_{*}=m_{\psi,t}$.
Note that if $t$ divides $t'$ then $\psi$ contains $f$ up to
$H_{t'}$ and $m_{\psi,t'}\leq m_{\psi,t}$ (because $H_{t}\leq H_{t'}$)
so $m_{\psi,t'}=m_{*}$. In addition, if $\psi'$ contains $f$ up
to $H_{t}$ and implies $\psi$ up to $H_{t}$ in the sense that $\psi\left(G\right)+\left(\left\{ 0\right\} \times H_{t}\right)\supseteq\psi'\left(G\right)$
(i.e., $\psi'\left(x,y\right)\to\exists e\in H_{t}\psi\left(x,y+e\right)$
holds), $m_{\psi',t}\leq m_{\psi,t}$ so $m_{\psi',t}=m_{*}$. 

Fix some $\psi_{0},t$ such that $m_{*}=m_{\psi_{0},t}$.

Let $H:=H_{t}$ and let $H':=H_{K!t}$. By ({*}), $m_{\psi_{0},H'}=m_{*}=m_{\psi_{0},H}$.

Let 
\[
\psi_{1}\left(x,y\right):=\psi_{0}\left(x,y\right)\land\forall x'\exists y'\exists e\in H\left(\psi_{0}\left(x',y'\right)\land\psi_{0}\left(x+x',y+y'+e\right)\right).
\]

Note that $\psi_{1}$ contains $f$ up to $H$: if $f\left(g\right)=h$,
then, for some $e\in H_{t}=H$, $\psi_{0}\left(g,h+e\right)$ and for
all $g'$, let $h':=f\left(g'\right)$. Then for some $e',e''\in H_{t}$,
$\psi_{0}\left(g',h'+e'\right)$ and $\psi_{0}\left(g+g',h+h'+e''\right)$.
So 
\[
M\models\psi_{0}\left(g,h+e\right)\land\psi_{0}\left(g',h'+e'\right)\land\psi_{0}\left(g+g',h+e+h'+e'+\left(e''-e-e'\right)\right)
\]
 so $\psi_{1}\left(g,h+e\right)$ holds. By ({*}), $m_{\psi_{1},H}=m_{*}$.

Let $g_{*}$ be such that $m_{*}=\left|\psi_{1}\left(g_{*},M\right)/H\right|$.
Enumerate it as $\set{f\left(g_{*}\right)+h_{i}+H}{i<m_{*}}$ where
$h_{0}\in H$.

Let 
\[
\psi_{2}\left(x,y\right):=\bigwedge_{1\leq k<K+1}\exists e_{k,0}e_{k,1}\in H\psi_{0}\left(kx,ky+e_{k,0}\right)\land\psi_{0}\left(g_{*}+kx,f\left(g_{*}\right)+ky+e_{k,1}\right).
\]
Then $\psi_{2}$ also contains $f$ up to $H$ (by a similar argument
as above) and implies $\psi_{0}$ up to $H$, so $m_{\psi_{2},H'}=m_{\psi_{2},H}=m_{*}$
by ({*}). Let $g_{\dagger}$ be such that $m_{*}=\left|\psi_{2}\left(g_{\dagger},M\right)/H'\right|$
and let $\set{f\left(g_{\dagger}\right)+c_{i}+H'}{i<m_{*}}$ be an
enumeration where $c_{0}\in H'$. Then it also follows that $m_{*}=\left|\psi_{2}\left(g_{\dagger},M\right)/H\right|$
and that $\set{f\left(g_{\dagger}\right)+c_{i}+H}{i<m_{*}}=\psi_{2}\left(g_{\dagger},M\right)/H$.

Note the following points:
\begin{enumerate}
\item For all $1\leq k\leq K$, $\psi_{0}\left(kg_{\dagger},M\right)/H=\set{k\left(f\left(g_{\dagger}\right)+c_{i}\right)+H}{i<m_{*}}$.
Why? as for all $i<m_{*}$, for some $e\in H$, $\psi_{2}\left(g_{\dagger},f\left(g_{\dagger}\right)+c_{i}+e\right)$
holds, we get $\supseteq$ by the definition of $\psi_{2}$. On
the other hand for $i\neq j$, $kc_{i}-kc_{j}=k\left(c_{i}-c_{j}\right)\notin H$
because otherwise $tk\left(c_{i}-c_{j}\right)=0$ so $c_{i}-c_{j}\in H'$
which is false since they belong to different cosets. This implies
that the set on the right has size $m_{*}$ and as $m_{*}=m_{\psi_{0},t}=\max\set{\left|\psi_{0}\left(g,M\right)/H\right|}{g\in G}$,
we get $\subseteq$ and equality.
\item For all $1\leq k\leq K$, $\psi_{0}\left(g_{*}+kg_{\dagger},M\right)/H=\set{f\left(g_{*}\right)+k\left(f\left(g_{\dagger}\right)+c_{i}\right)+H}{i<m_{*}}$
(for the same reason as (1)).
\item For $i<m_{*}$ and $1\leq k\leq K$, as $\psi_{1}\left(g_{*},f\left(g_{*}\right)+h_{i}+e\right)$
holds for some $e\in H$, applying the second half of $\psi_{1}$
with $x'=kg_{\dagger}$, by (1) we obtain some $j=j_{i,k}<m_{*}$ such
that $y'=k\left(f\left(g_{\dagger}\right)+c_{j}\right)+e'$ for some
$e'\in H$ works: $\psi_{0}\left(g_{*}+kg_{\dagger},M\right)/H$ contains
$f\left(g_{*}\right)+h_{i}+k\left(f\left(g_{\dagger}\right)+c_{j}\right)+H$.
\item By (2) and (3) we obtain that, for all $i<m_{*}$ and $1\leq k\leq K$
there is some $j'=j'_{i,k}$ such that $f\left(g_{*}\right)+k\left(f\left(g_{\dagger}\right)+c_{j'}\right)-f\left(g_{*}\right)-h_{i}-k\left(f\left(g_{\dagger}\right)+c_{j}\right)\in H$
which implies 
\[
k\left(c_{j'}-c_{j}\right)\in h_{i}+H.
\]
\end{enumerate}
Assume towards contradiction that $m_{*}>1$. As $K=m^{2}+1\geq m_{*}^{2}+1$,
for $i=1$, we get $1\leq k<k'\leq K$ such that $j=j_{1,k}=j_{1,k'}$
and $j'=j'_{1,k}=j'_{1,k'}$. As $h_1 \notin H$, $j \neq j'$. By (4), $k'\left(c_{j'}-c_{j}\right)-k\left(c_{j'}-c_{j}\right)\in H$
so that $\left(k'-k\right)\left(c_{j'}-c_{j}\right)\in H$, i.e.,
$t\left(k'-k\right)\left(c_{j'}-c_{j}\right)=0$, so $c_{j'}-c_{j}\in H'$,
contradiction.
\end{proof}

\begin{exmp}
\label{example_showing_saturation_is_required}
The following exemplifies the need for saturation.  

Let $N:=\left(\Q\left(\pi\right),f_{\pi},+\right)$ and $M:=\left(\Q\left(\pi\right),+\right)$
where $f_{\pi}:\Q\left(\pi\right)\to\Q\left(\pi\right)$ is multiplication
by $\pi$. Let $M':=M_{M},N':=N_{N}$ (i.e., the structures we get by naming all elements). Clearly, $M'$ is a $0$-reduct of $N'$.
Also, since $\dcl_{M'} \left( \emptyset \right) = \acl_{M'} \left( \emptyset \right) = \Q \left( \pi \right)$, we have $\dcl_{N'} \subseteq \acl_{M'}$.

However, $f_{\pi}$ is not definable in $M'$:  Let $N^{*}\succ N'$
be a proper extension of $N'$. Let $e\in N^{*}\backslash N'$. Let
$M''$ be the vector space generated by $\Q\left(\pi\right)$ and
$e$ over $\Q$. Then $M''\succ M'$ by quantifier elimination for
divisible abelian groups. However, $f_{\pi}\left(e\right)\notin M''$
(otherwise $N^{*}\models f_{\pi}\left(e\right)=v+qe$ where $v\in\Q\left(\pi\right)$
and $q\in\Q$, but $v/\left(\pi-q\right)\in N$ and hence for some
$e'\in N'$ such that $e'\neq v/\left(\pi-q\right)$, $N'\models\pi e'=v+qe'$,
a contradiction). Hence, it cannot be that $f_{\pi}$ is definable
in $M'$ (suppose that $\varphi\left(x,y\right)$ defines it. Let
$M^{*}$ be the reduct of $N^{*}$ to the language of $M$, so that
$M''\prec M^{*}$ and $\varphi^{M^{*}}=\varphi^{N^{*}}$. Then $N^{*}\models\forall x\forall y\varphi\left(x,y\right)\leftrightarrow f_{\pi}\left(x\right)=y$,
and since $M''\models\forall x\exists y\varphi\left(x,y\right)$,
it follows that if $M''\models\varphi\left(e,e'\right)$ then $M^{*}\models\varphi\left(e,e'\right)$
and hence $e'=f_{\pi}\left(e\right)$, a contradiction.)
\end{exmp}

\begin{exmp}
\label{example_showing_0_expansion_is_required}
The following shows that in \cref{lem_homomorphism_preservation_R_groups,lem_homomorphism_preservation_small_quotients} it is not enough to assume that $N$ is just an expansion of $M$, instead of a $0$-expansion.
Continuing \cref{example_showing_saturation_is_required}, let $\cN \succ N$ be $\omega$-saturated, and let $\cM \succ M$ be the reduct of $\cN$ to the language of $M$. Let $\cM'$ be the structure obtained from $\cM$ by naming all elements. Then $\dcl_{\cN} \subseteq \acl_{\cM'}$. However, $f_{\pi}^{\cN}$ is not definable in $\cM'$ (as otherwise, $f_{\pi}$ would be definable in $M$).
\end{exmp}

Combining \cref{lem_homomorphism_preservation_small_quotients} and \cref{Loveys_on_weakly_minimal_one_based} we get:

\begin{prop}
\label{corr_of_Loveys_and_preservation_of_homomorphisms}
Suppose that $M$ is a $0$-reduct of an $\omega$-saturated structure $N$, both $0$-expanding an abelian group $\left( G, + \right)$ with small quotients.
Suppose that: 
\begin{enumerate}
\item $\dcl_{N} \subseteq \acl_{M}$, and 
\item $N$ is weakly-minimal and 1-based.
\end{enumerate}
Then $M$ and $N$ are interdefinable.
\end{prop}

\begin{proof}
The statement is clear when $M$ is finite, so we may assume it is infinite. By \cref{lem_homomorphism_preservation_small_quotients}, every endomorphism of the group $\left( G, + \right)$ that is definable in $N$ is definable in $M$. Also note that since $\left( G, + \right)$ has small quotients, it has unbounded exponent. The conclusion now follows from \cref{Loveys_on_weakly_minimal_one_based}
\end{proof}

\section{Proof of \texorpdfstring{\cref{main_theorem_1_intro}}{\ref{main_theorem_1_intro}}} \label{sec:main theorem}

\begin{thm}
\label{main_theorem}
Let $M$ be a weakly-minimal and 1-based structure expanding an abelian group $\left( G,+ \right)$ with small quotients.
Let $N$ be a $0$-expansion of $M$ such that:
\begin{enumerate}
\item $N$ is $\omega$-saturated.
\item Every unary subset $X \subseteq G$ that is definable in $N$ is also definable in $M$.
\item $\acl_M = \acl_N$.
\end{enumerate}
Then $M$ and $N$ are interdefinable.
\end{thm}

\begin{proof}

The group operation is definable in $M$, possibly using finitely many parameters. By naming these parameters in $M$ and $N$, we may assume that $M$ $0$-expands $\left( G, \cdot \right)$. Since we named only finitely many elements, $N$ remains $\omega$-saturated.

We show that it is enough to prove the theorem when $N$ is a monster model for $\Th{N}$. Thus, let $\cN$ be a monster model for $\Th{N}$, and let $\cM := \cN \restriction L$. Clearly, $\cM$ is a weakly-minimal and 1-based structure $0$-expanding an abelian group $\left( \cG,+ \right)$. By \cref{small_quotients_or_torsion_are_elementary_properties}, $\cG$ has small quotients. By \cref{obs_not_adding_unary_subsets_enough_to_check_for_saturated_enough}, every unary subset $X \subseteq \cG$ that is definable in $\cN$ is also definable in $\cM$. By \cref{obs_equality_of_acl_operators_enough_to_check_for_saturated_enough}, $\acl_{\cM} = \acl_{\cN}$. So, by the theorem for the case of monster models, we find that $\cM$ and $\cN$ are interdefinable. By \cref{interdefinability_passes_to_elementary_substructures}, $M$ and $N$ are interdefinable.

So, without loss of generality, we may assume that $N$ is a monster model for $\Th{N}$. 
Let $N_0 \prec N$ be a small model and denote by $M'$ and $N'$ the structures obtained from $M$ and $N$ by adding constants for all elements of $N_0$. All assumptions still hold and $N'$ is still a monster model for its theory. So we may assume that $M=M'$ and $N=N'$.
Since every unary subset $X \subseteq G$ which is definable in $N$ is also definable in $M$, and since $M$ is weakly-minimal, by \cref{weakly_minimal_equiv_conditions} (\ref*{weakly_minimal_equiv_conditions_intersection_every_model}) we get that $N$ is also weakly-minimal. Since $\acl_N \left( \emptyset \right) = N_0$ is a model of $\Th{N}$, by \cref{weakly_minimal_and_acl0_is_a_model_implies_wei} $\Th{N}$ has WEI and hence also GEI. Similarly, $\acl_M \left( \emptyset \right) = N_0 \restriction L \prec M$, hence $\Th{M}$ has GEI.

Denote by $\ind^M$ and $\ind^N$ the ternary relations of forking independence in $M,N$, respectively, and denote by $\ind^{M,a}$ and $\ind^{N,a}$ the ternary relations of $\acl$-independence in $M,N$, respectively (see \cref{def:acl-independence}). We claim that for every (real) tuples $a,b \in G$ and $C \subseteq G$, $a \ind^{M}_C b \iff a \ind^{N}_C b$. Indeed, since both $M$ and $N$ are weakly minimal, by \cref{u_rank_1_acl_facts} (\ref*{u_rank_1_forking_indep_coincides_with_acl_indep}) we have 
$a \ind^{M}_C b \iff a \ind^{M,a}_C b$ 
and 
$a \ind^{N}_C b \iff a \ind^{N,a}_C b$.
But $\acl_M = \acl_N$, hence $a \ind^{M,a}_C b \iff a \ind^{N,a}_C b$.

It now follows that $\Th{N}$ is 1-based: Since $\Th{M}$ is 1-based and has GEI, by \cref{equivalent_cond_for_1_based_under_gei} for every (real) tuples $a,b \in G$ we have 
$a \underset{\acl_M \left( a \right) \cap \acl_M \left( b \right)}{\ind^{\mathclap{M}}} b$. But $\acl_M = \acl_N$, so 
$a \underset{\acl_N \left( a \right) \cap \acl_N \left( b \right)}{\ind^{\mathclap{M}}} b$, and by the previous paragraph we obtain 
$a \underset{\acl_N \left( a \right) \cap \acl_N \left( b \right)}{\ind^{\mathclap{N}}} b$. Since $\Th{N}$ has GEI, by \cref{equivalent_cond_for_1_based_under_gei} $\Th{N}$ is 1-based.

Now all the assumptions of \cref{corr_of_Loveys_and_preservation_of_homomorphisms} are satisfied, therefore $M$ and $N$ are interdefinable.
\end{proof}

\begin{rem}
In \cref{main_theorem}, if $N$ is $\pipes{L}^{+}$-saturated, where $L$ is the language of $M$, then it is enough to assume that $N$ is an expansion of $M$ instead of a $0$-expansion.
\end{rem}

\begin{proof}
There exists a subset $A \subseteq G$ of size $\pipes{A} \le \pipes{L}$ such that every set that is $\emptyset$-definable in $M$ is also $A$-definable in $N$.
Denote by $M'$ and $N'$ the structures obtained from $M$ and $N$ by adding constants for all elements of~$A$. Then $N'$ is also $\pipes{L}^{+}$-saturated, and all the assumptions of \cref{main_theorem} still hold for $M'$ and $N'$. So by \cref{main_theorem} $M'$ and $N'$ are interdefinable and therefore $M$ and $N$ are interdefinable.
\end{proof}

Note that if $M = \roundedb{G,+}$ is itself an abelian group with small quotients and small torsion, then (by \cref{pure_abelian_group_is_1_based}) it is already 1-based and (by \cref{abelian_group_with_small_quotients_and_small_torsion_is_weakly_minimal}) weakly-minimal.

By \cref{for_ab_group_of_fin_U_rk_small_torsion_equiv_small_quotients} we get:

\begin{cor}
\label{main_theorem_for_small_torsion}
\cref{main_theorem} holds if we assume that $\left( G,+ \right)$ has small torsion rather than small quotients.
\end{cor}

By \cref{interdefinability_passes_to_elementary_substructures} we get:

\begin{cor}
\label{cor_main_theorem_passing_to_elementary_extensions}
Let $M$ be a structure and $N$ be a $0$-expansion of $M$. Suppose that there are elementary extensions $M' \succ M$ and $N' \succ N$ that satisfy the assumptions of \cref{main_theorem} or \cref{main_theorem_for_small_torsion}. Then $M$ and $N$ are interdefinable.
\end{cor}

\begin{exmp}
The structures $M'$ and $N'$ from \cref{example_showing_saturation_is_required} show that some saturation is required in \cref{main_theorem}. Both $M'$ and $N'$ are strongly-minimal: therefore, every unary subset that is definable in $N'$ is also definable in $M'$. The remaining requirements of \cref{main_theorem} (except for saturation) are clearly satisfied.
\end{exmp}
The following example shows that \cref{main_theorem} is not true if we remove the assumption of small quotients.
\begin{exmp} \label{exa:need small quotients}
    Let $F$ be the field $\mathbb{F}_2$ with two elements. Let $\alpha$ satisfy the (irreducible) polynomial equation $X^2 + X + 1 =0$, and let $K = F\roundedb{\alpha}$ (a fields with 4 elements: $0,1,\alpha,1+\alpha$). Let $V$ be an infinite countable vector space over $K$, and fix some nonzero $c\in V$. 
    Let $N$ be the structure $\roundedb{V,+,\lambda_\alpha,c}$ where $\lambda_{\alpha}\roundedb{v}=\alpha v$ (scalar multiplication). 
    Let $M = \roundedb{V,+,U,c}$ where $U$ is a binary relation defined by: 
    $U\roundedb{u,v}$ iff 
    $v \in \spn[F]{\alpha u,c,\alpha c}$.
    Then $M$ expands an abelian group $\roundedb{V,+}$, is a 0-reduct of $N$ and all three structures are $\omega$-categorical (so saturated) and strongly-minimal (so weakly minimal). Additionally, by \cite[Chapter 2, Theorem 5.12]{PillayGST} or \cite[Lemma 4.3, Theorem 4.2]{BaysGST}, since $M$ is both strongly-minimal and $\omega$-categorical, $M$ is 1-based (alternatively, by \cite[Chapter 4, Proposition 6.3]{PillayGST}, since $M$ is a reduct of a 1-based structure of finite $U$-rank, $M$ is 1-based). Finally, for any $a\in V$, $\acl_M\roundedb{a}$ contains $\alpha a$ (since it is algebraic over $a$ because of the choice of $U$), so (since $c$ is named as a constant) for any set $A$, $\acl_M\roundedb{A}$ contains $\spn[F]{Ac \cup \alpha \roundedb{Ac}} = \spn[K]{Ac}=\acl_N(A)$. The other containment is clear so we have $\acl_M = \acl_N$ as required. 
    However, $\lambda_\alpha$ is not definable in $M$ as we now show. 

    Assume towards contradiction that $\lambda_\alpha$ is definable in $M$. Let $B$ be a basis for $V$ over $K$ containing $c$, so that $C:=B \cup \alpha B$ is a basis for $V$ over $F$. Since $\lambda_\alpha$ is definable over some finite set of parameters, there exists a finite set $c \in B_0 \subseteq B$ such that $\lambda_{\alpha}$ is definable over $B_0 \cup \alpha B_0$. Choose $a \in B \setminus B_0$. Note that $\curlyb{\alpha a +c} \cup C \setminus \curlyb{\alpha a}$ is still a basis for $V$ over $F$. Let $T:V\to V$ be an automorphism of $V$ as an $F$-vector space fixing $C\setminus \curlyb{\alpha a}$ while $T\roundedb{\alpha a} = \alpha a + c$. We claim that $T$ is an automorphism of $M$. This will suffice since $T$ clearly does not preserve $\lambda_\alpha$. 

   To show that $T$ is an automorphism of $M$, notice that it fixes $c$ and preserves the addition operation, so we are left to show that it preserves $U$. Note that (*) for any $u \in V$, if $u$ does not contain $\alpha a$ in its (unique) representation as a linear combination of elements of $C$ over $F$, then $T\roundedb{u} = u$. Otherwise, $T\roundedb{ u}=u +c$. Suppose that $U\roundedb{u,v}$ holds. Then $v \in S:=\spn[F]{\alpha u,c,\alpha c}$. Since $T$ fixes $c,\alpha c$, it follows from (*) that $T\roundedb{v}\in T\roundedb{S} = S$. On the other hand, (*) implies that $\spn[F]{\alpha T\roundedb{u},c, \alpha c}=S$ as well. Together, it follows that $T\roundedb{v} \in T\roundedb{S} = S = \spn[F]{\alpha T\roundedb{u},c, \alpha c}$ which implies that $U\roundedb{T\roundedb{u},T\roundedb{v}}$ holds. As $T = T^{-1}$, it follows that if $U\roundedb{T\roundedb{u},T\roundedb{v}}$ holds then also $U\roundedb{u,v}$ holds and we are done.
\end{exmp}

\begin{question} \label{que:fields}
Does \cref{main_theorem} holds when we replace a group by a field? Namely, suppose that in the context of \cref{main_theorem}, we remove the assumption that $M$ is 1-based and expands a group with small quotients and replace it with the assumption that $M$ expands a field (which will have to be algebraically closed), but keep all the other assumptions. Does the conclusion still hold?
\end{question}

This seems to be open even if we assume that $M$ is strongly-minimal.%
\footnote{Cf. the discussion in \url{https://mathoverflow.net/q/487997/101562}.}
At least in the case when $M$ is a (pure) algebraically closed field (so strongly-minimal), the answer is positive.
\begin{fact}[{\cite[Theorem 1]{Hrushovski1992}}]
\label{hrushovski_theorem_about_strongly_minimal_expansions_of_fields}
Let $M$ be an algebraically closed field (in the language of rings), and let $N$ be a strongly-minimal $0$-expansion of $M$. Denote $K := \acl_{N} \roundedb{0}$, and suppose that for all $A \subseteq M$, $\acl_{M} \roundedb{A \cup K} = \acl_{N} \roundedb{A \cup K}$. Then $M$ and $N$ are interdefinable.
\end{fact}


\section{Elementary extensions of archimedean ordered abelian groups of finite rank} \label{sec:elementary extentions}

In this section, we apply \cref{main_theorem} in the context of ordered abelian groups to prove \cref{main_theorem_2_intro} (see \cref{thm_application_small_quotients}).

\begin{lem}
\label{lem_extracting_an_interval}
Let $N = \roundedb{G,+,<}$ be elementarily equivalent to an archimedean ordered abelian group with small quotients, and let $M$ be a reduct (but not necessarily a $0$-reduct) of $N$ that expands $\roundedb{G,+}$. Suppose that there exists a unary subset $D \subseteq G$ that is definable in $M$ but not in $\roundedb{G,+}$. Then, there exists $b \in G \cup \curlyb{\infty}$ such that the interval $\roundedb{0,b}$ is infinite and definable in $M$.
\end{lem}

\begin{proof}
By \cref{cor_shape_of_unary_subsets_archimedean_small_quotients}, $D = \bigcup_{i=1}^N D_i$ such that, for each $i$, either
\begin{enumerate}

\item $D_i = \curlyb{c_i}$ is a singleton, or

\item $D_i = I_i \cap \roundedb{mG + g_i}$, where $g_i \in G$ and $I_i$ is an open generalized interval with rational endpoints,

\end{enumerate}
and $1 \le m < \omega$ does not depend on $i$. 

By removing at most finitely many points from $D$, we may assume that, for all $i$, $D_i$ is infinite, and, in particular, it is of the latter form.
Let $h_1, \dots, h_K \in G$ be such that for every $1 \le i \le N$ there exists exactly one $1 \le j \le K$ such that $g_i \equiv_m h_j$. We can rewrite $D$ as 
\[
D = \bigcup_{j=1}^K \roundedb{ \roundedb{mG + h_j} \cap \bigcup_{i=1}^{N_j} I_{j,i} }
\]
where each $I_{j,i}$ is an infinite open generalized interval with rational endpoints and moreover the intersection $\roundedb{ \roundedb{mG + h_j} \cap \bigcup_{i=1}^{N_j} I_{j,i}}$ is infinite.

Since $D$ is not definable in $\roundedb{G,+}$, there exists $j_0$ such that $\roundedb{mG + h_{j_0}} \cap \bigcup_{i=1}^{N_{j_0}} I_{j_0,i}$ is not definable in $\roundedb{G,+}$. By replacing $D$ with 
\[
D \cap \roundedb{mG + h_{j_0}} = \roundedb{mG + h_{j_0}} \cap \bigcup_{i=1}^{N_{j_0}} I_{j_0,i}
\]
(which is clearly definable in $M$), we may assume that $D$ has the following form 
\[
D = \roundedb{mG + g} \cap \bigcup_{i=1}^N I_i
\]
where each $I_i$ is an infinite open generalized interval with rational endpoints.

By replacing $D$ with $D - g$, we may assume that $g=0$, we may assume that $D = mG \cap \bigcup_{i=1}^N I_i$ where each intersection $mG \cap I_i$ is infinite. For each $i$, let $J_i := I_i / m = \set{a \in G}{ma \in I_i}$. Then $J_i$ is an infinite open generalized interval with rational endpoints. Since $D/m = \bigcup_{i=1}^N J_i$ is definable using only $D$ and $+$, and vice versa, $D/m$ is definable in $M$ but not in $\roundedb{G,+}$. Therefore, by replacing $D$ with $D/m$, we may assume that $D$ is of the form $D = \bigcup_{i=1}^N J_i$, where each $J_i$ is an infinite open generalized interval with rational endpoints.

Whenever two intervals $J_{i_1}$ and $J_{i_2}$ have a nonempty intersection, we may replace them with their union, which is also an infinite open generalized interval with rational endpoints. Thus, we may assume that the intervals $\curlyb{J_i}_{i=1}^N$ are pairwise disjoint.
Note that $G \backslash D$ is a finite union of disjoint (not necessarily open) generalized intervals with rational endpoints. By removing at most finitely many points from $G \backslash D$, we obtain a set $D'$ which is a finite union of disjoint infinite open generalized intervals with rational endpoints, and which is definable in $M$ but not in $\roundedb{G,+}$. If $D$ is not bounded from below, $D'$ is bounded from below. Thus, by replacing $D$ with $D'$ if necessary, we may assume that $D$ is bounded from below.

To recap, $D = \bigcup_{i=1}^N J_i$, where each $J_i$ is an infinite open generalized interval with rational endpoints, $D$ is bounded from below, and $D$ is definable in $M$ but not in $\roundedb{G,+}$. By reordering, we may assume that $ i < j  \implies  J_i < J_j$. Since $D$ is not definable in $\roundedb{G,+}$, it is not empty, i.e., $N \ge 1$. Since $D$ is bounded from below, so is $J_1$. There are two cases:

If $N$ is discrete, then (since it is elementarily equivalent to an archimedean ordered abelian group) we may assume that it is an elementary extension of $\roundedb{ \Z,+,< }$. Since $J_1$ is definable in $N$, it has a minimum, which we denote by $a$. By replacing $D$ with $D - a + 1$, we may assume that $a=1$. If $J_1$ is not bounded from above, then $D = J_1 = \roundedb{0, \infty}$ is definable in $M$. Otherwise, $J_1$ has a maximum, which we denote by $b$. Now $D \cap \roundedb{-D + b} = \roundedb{0,b}$ is infinite and definable in $M$.

So suppose that $N$ is dense. $J_1$ is an infinite open generalized interval with rational endpoints bounded from below, so we can write it as $J_1 = \roundedb{\frac{a}{n},\frac{b}{n}}$ with $1 \le n < \omega$, $a \in G$, $b \in G \cup \curlyb{\infty}$, and $a < b$. If $b = \infty$, let $c \in J_1$. Otherwise, by \cref{lem_intersection_of_interval_with_coset} we have $\roundedb{2a,a+b} \cap 2nG \neq \emptyset$, hence $\roundedb{ \frac{a}{n} , \frac{a+b}{2n} } \neq \emptyset$, so let $c \in \roundedb{ \frac{a}{n} , \frac{a+b}{2n} }$.

Let $E := \roundedb{D - c} \cap \roundedb{-D + c}$. Then $E$ is definable in $M$, and, denoting $d := c - \frac{a}{n}$, we have $E = \roundedb{-d,d}$. Note that $d$ is in the divisible hull of $G$, but not necessarily in $G$. For $g \in E$ denote $C_g := E \cap \roundedb{E - g}$. Note that $\set{C_g}{g \in E}$ is uniformly definable in $M$. Let $e \in \roundedb{0,d}$, and let
\[
E_2 := \set{0 \neq g \in E}{C_g \supsetneq C_e}.
\]
Then $E_2$ is definable in $M$. Note that for $g \in E$, if $g \ge 0$, then $C_g = \roundedb{-d , d - g}$, and if $g < 0$ then $C_g = \roundedb{-d - g , d}$. Therefore $E_2 = \roundedb{0,e}$. So $E_2$ is as required.
\end{proof}

\begin{fact}[{\cite[Lemma 5.13]{Alouf_dElbee_19}}]
\label{fact_unstable_implies_new_unary_def_set_in_monster}
Let $\cM$ be a monster model of an unstable theory $T$, and let $\cM^{-}$ be a $0$-reduct of $\cM$ that is stable. Then, there exists a unary subset $\cD \subseteq \cM$ that is definable in $\cM$ but not in $\cM^{-}$.
\end{fact}

\begin{fact}[{\cite[Claim 5.16]{Alouf_dElbee_19}}]
\label{fact_detecting_intervals}
Let $N = \roundedb{G,+,<}$ be an elementary extension of $\roundedb{\Z,+,<}$, and let $\varphi \roundedb{x,z}$ be a formula without parameters.
Let $\chi \roundedb{y,z}$ be the formula $\chi_{1} \roundedb{y,z} \wedge \chi_{2} \roundedb{y,z} \wedge \chi_{3} \roundedb{y,z}$ where:
\begin{itemize}

\item $\chi_{1}$ is $\varphi \roundedb{0,z} \wedge \varphi \roundedb{y,z} \wedge \neg \varphi \roundedb{-1,z} \wedge \neg \varphi \roundedb{y+1,z} \wedge \neg \varphi \roundedb{2y,z}$, 

\item $\chi_{2}$ is $\forall w \roundedb{ \roundedb{w \neq 0 \wedge \varphi \roundedb{w,z} } \rightarrow \varphi \roundedb{w-1,z} }$,

\item $\chi_{3}$ is $\forall w \roundedb{ \roundedb{ w \neq y \wedge \varphi \roundedb{w,z} } \rightarrow \varphi \roundedb{w+1,z} }$.
\end{itemize}
Then, for every $b,c \in G$, $N \models \chi \roundedb{b,c}$ if and only if $b > 0$ and $\varphi \roundedb{G,c} = \squareb{0,b}$. 

\end{fact}

\begin{prop}
\label{prop_equiv_between_stability_and_no_new_unary_def_subsets_and_no_new_unary_def_subsets_in_monster}
Let $N = \roundedb{G,+,<}$ be elementarily equivalent to an archimedean ordered abelian group with small quotients, and let $M$ be a reduct (but not necessarily a $0$-reduct) of $N$ that expands $\roundedb{G,+}$. Let $\cN = \roundedb{\cG,+,<}$ be a monster model for $\Th{N}$, and let $\cM$ be the reduct of $\cN$ to the language of $M$. Then the following are equivalent:
\begin{enumerate}

\item $M$ is stable.

\item Every unary subset $D \subseteq G$ that is definable in $M$ is also definable in $\roundedb{G,+}$.

\item Every unary subset $\cD \subseteq \cG$ that is definable in $\cM$ is also definable in $\roundedb{\cG,+}$.

\end{enumerate}

\end{prop}

\begin{proof}
$\roundedb{3} \implies \roundedb{2}$ is clear, and $\roundedb{3} \implies \roundedb{1}$ is by \cref{fact_unstable_implies_new_unary_def_set_in_monster} (after naming elements from $\cM$ to make it a $0$-expansion of $\roundedb{\cG,+}$). We show $\roundedb{1} \implies \roundedb{3}$ and $\roundedb{2} \implies \roundedb{3}$.

For $\roundedb{1} \implies \roundedb{3}$: Suppose that there exists a unary subset $\cD \subseteq \cG$ that is definable in $\cM$ but not definable in $\roundedb{\cG,+}$. By \cref{lem_extracting_an_interval}, there exists $b \in \cG \cup \curlyb{\infty}$ such that the interval $I := \roundedb{0,b}$ is infinite and definable in $\cM$. Let $R \roundedb{x_1,x_2}$ be the relation defined by $x_2 - x_1 \in I$. Then $R$ is definable in $\cM$, and on $I$ the relation $R$ coincides with the order $<$. Since $I$ is infinite, $\cM$ is unstable.

For $\roundedb{2} \implies \roundedb{3}$: Suppose that there exists a unary subset $\cD \subseteq \cG$ that is definable in $\cM$ but not definable in $\roundedb{\cG,+}$. By \cref{lem_extracting_an_interval}, there exists $b \in \cG \cup \curlyb{\infty}$ such that the interval $I := \roundedb{0,b}$ is infinite and definable in $\cM$. Let $\varphi \roundedb{x,z}$ be a formula without parameters in the language of $\cM$ and let $c$ be a tuple from $\cG$ such that $\varphi \roundedb{\cG,c} = I$. There are two cases:

If $N$ is discrete, then (since it is elementarily equivalent to an archimedean ordered abelian group) we may assume that it is an elementary extension of $\roundedb{ \Z,+,< }$. If $b = \infty$, let $\theta \roundedb{z}$ be the formula 
\[
\forall x \roundedb{ \varphi \roundedb{x,z} \leftrightarrow x > 0}.
\]
Then $c$ satisfies $\theta \roundedb{z}$, hence by elementarity there exists a tuple $e$ in $G$ that satisfies $\theta \roundedb{z}$. So $\varphi \roundedb{G,e} = \roundedb{0,\infty}$ is definable in $M$ but not definable in $\roundedb{G,+}$ by stability as in the proof of $\roundedb{1} \implies \roundedb{3}$, as required.

So suppose that $b < \infty$. Replace $I = \roundedb{0,b}$ with $I := \squareb{0,b}$ (with $\varphi \roundedb{x,z}$ as before, but with respect to this new $I$). Let $\chi \roundedb{y,z}$ be as in \cref{fact_detecting_intervals}. So $\chi \roundedb{y,z}$ is a formula in the language of $\cM$, and for every $b',c' \in \cG$, $\cN \models \chi \roundedb{b',c'}$ if and only if $b' > 0$ and $\varphi \roundedb{\cG,c'} = \squareb{0,b'}$. 
Let $\psi \roundedb{x}$ be the following formula
\[
\exists y,z \roundedb{ \chi \roundedb{y,z} \wedge \varphi \roundedb{x,z} }.
\]
Then $\psi \roundedb{x}$ is a formula in the language of $\cM$. We show that $\psi \roundedb{G}$ is an infinite convex set and that $0$ is its minimum element. Clearly, if $a \in \psi \roundedb{G}$ then $a \ge 0$, and if $\psi \roundedb{G} \neq \emptyset$ then $0 \in \psi \roundedb{G}$. If $a_1 < a_2 < a_3$ and $a_1, a_3 \in \psi \roundedb{G}$, then, first, $a_1 \ge 0$ and hence $a_2 \ge 0$, and second, there are $b',c' \in G$ such that $\varphi \roundedb{G,c'} = \squareb{0,b'}$ and $a_3 \in \squareb{0,b'}$. In particular, $a_3 \le b'$, so we have $0 \le a_2 \le b'$, i.e., $a_2 \in \squareb{0,b'}$, and therefore $a_2 \in \psi \roundedb{G}$. It remains to show that $\psi \roundedb{G}$ is infinite. For each $n < \omega$, since $\varphi \roundedb{\cG,c} = I = \squareb{0,b}$ is infinite, we have 
\[
\cM \models \chi \roundedb{b,c} \wedge \exists^{\ge n} x \varphi \roundedb{x,c}. 
\]
By elementarity, there are $d_n, e_n \in G$ such that 
\[
M \models \chi \roundedb{d_n,e_n} \wedge \exists^{\ge n} x \varphi \roundedb{x,e_n}, 
\]
hence $\psi \roundedb{G}$ has at least $n$ elements. Therefore $\psi \roundedb{G}$ is infinite.
It follows that $\psi \roundedb{G}$ is not definable in $\roundedb{G,+}$, as required: Otherwise, the same proof as in $\roundedb{1} \implies \roundedb{3}$ shows that $\roundedb{G,+}$ is unstable, a contradiction.

So suppose that $N$ is dense. Recall that $\varphi \roundedb{\cG,c} = I = \roundedb{0,b}$. Let $\theta \roundedb{z}$ be the following formula 
\[
\exists y \roundedb{ \roundedb{y > 0} \wedge \forall x \roundedb{ \varphi \roundedb{x,z} \leftrightarrow 0 < x < y} }.
\]
Then $c$ satisfies $\theta \roundedb{z}$, hence by elementarity there exists a tuple $e$ in $G$ that satisfies $\theta \roundedb{z}$. So there exists $0 < d \in G$ such that $J := \varphi \roundedb{G,e} = \roundedb{0,d}$. Since $d > 0$ and $N$ is dense, $J$ is infinite. It follows that $J$ is not definable in $\roundedb{G,+}$ (again by the proof of $\roundedb{1} \implies \roundedb{3}$), as required.
\end{proof}

\begin{thm}
\label{thm_application_small_quotients}
Let $N = \roundedb{G,+,<}$ be elementarily equivalent to an archimedean ordered abelian group with small quotients, and let $M$ be a reduct (but not necessarily a $0$-reduct) of $N$ that expands $\roundedb{G,+}$. Suppose that one of the following equivalent\footnote{The equivalence follows from \cref{prop_equiv_between_stability_and_no_new_unary_def_subsets_and_no_new_unary_def_subsets_in_monster}.} conditions holds:
\begin{enumerate}

\item $M$ is stable, or

\item every unary subset $D \subseteq G$ that is definable in $M$ is also definable in $\roundedb{G,+}$.

\end{enumerate}
Then $M$ is interdefinable with $\roundedb{G,+}$.

\end{thm}

\begin{proof}

Since every abelian group is 1-based (in the pure group language; see \cref{pure_abelian_group_is_1_based}), $\roundedb{G,+}$ is 1-based. By \cref{abelian_group_with_small_quotients_and_small_torsion_is_weakly_minimal}, $\roundedb{G,+}$ is also weakly-minimal.

Let $\cN = \roundedb{\cG,+,<}$ be a monster model for $\Th{N}$ which is at least $\pipes{G}^+$-saturated, and let $\cM$ be the reduct of $\cN$ to the language of $M$. Denote by $\cG_G$, $\cM_G$ and $\cN_G$ the structures obtained from $\roundedb{\cG,+}$, $\cM$ and $\cN$ by adding all the elements of $G$ as constants. Then $\cG_G$ is weakly-minimal and 1-based, and $\cM_G$ is $\pipes{G}^+$-saturated. Also note that now $\cN_G$ $0$-expands $\cM_G$ which $0$-expands $\roundedb{\cG,+}$. 
By \cref{prop_equiv_between_stability_and_no_new_unary_def_subsets_and_no_new_unary_def_subsets_in_monster}, every unary subset $\cD \subseteq \cG$ that is definable in $\cM$ (equivalently, in $\cM_G$) is also definable in $\roundedb{\cG,+}$ (equivalently, in $\cG_G$).

By \cref{cor_acl_archimedean_small_quotients}, for every $A \subseteq \cG$ we have $\acl_{\cN_G} \roundedb{A} = \acl_{\cN} \roundedb{A \cup G} = \acl_{\roundedb{\cG,+}} \roundedb{A \cup G} = \acl_{\cG_G} \roundedb{A}$ (note that if $N$ is discrete then $1 \in G$ 
 and this was required in \cref{cor_acl_archimedean_small_quotients}). Since $\cN_G$ $0$-expands $\cM_G$ which $0$-expands $\cG_G$, for every $A \subseteq \cG$ we have $\acl_{\cG_G} \roundedb{A} \subseteq \acl_{\cM_G} \roundedb{A} \subseteq \acl_{\cN_G} \roundedb{A}$. Therefore $\acl_{\cG_G} = \acl_{\cM_G} = \acl_{\cN_G}$.

Now all the assumptions of \cref{main_theorem} hold (with $\cG_G$ and $\cM_G$ taking the roles of $M$ and $N$ from \cref{main_theorem}, respectively), therefore, by \cref{cor_main_theorem_passing_to_elementary_extensions}, $M$ is interdefinable with $\roundedb{G,+}$.
\end{proof}

By \cref{fact_torsion_free_abelian_group_of_finite_rank_has_small_quotients} we get:

\begin{cor}
\label{cor_application_finite_rank}
\cref{thm_application_small_quotients} holds if instead of ``with small quotients'' we write ``of finite rank''.
\end{cor}

\begin{cor}
\label{cor_defining_the_order_on_an_interval}
Let $N = \roundedb{G,+,<}$ be elementarily equivalent to an archimedean ordered abelian group with small quotients, and let $M$ be a reduct (but not necessarily a $0$-reduct) of $N$ that is a proper expansion of $\roundedb{G,+}$. Then there exists  $b \in G \cup \curlyb{\infty}$ such that the interval $I := \roundedb{0,b}$ is infinite and definable in $M$ and the restriction of $<$ to $I$ is definable in $M$.
\end{cor}

\begin{proof}
Since $M$ is a proper expansion of $\roundedb{G,+}$, by \cref{thm_application_small_quotients} there exists a unary subset $D \subseteq G$ that is definable in $M$ but not definable in $\roundedb{G,+}$. 
By \cref{lem_extracting_an_interval}, there exists $b \in G \cup \curlyb{\infty}$ such that the interval $I := \roundedb{0,b}$ is infinite and definable in $M$. Let $R \roundedb{x_1,x_2}$ be the relation defined by $x_2 - x_1 \in I$. Then $R$ is definable in $M$, and on $I$ the relation $R$ coincides with the order $<$. 
\end{proof}

In particular, \cref{cor_defining_the_order_on_an_interval} holds for elementary extensions of $\roundedb{\Z,+,<}$ and $\roundedb{\Q,+,<}$. This answers positively a question of Conant, see \cite[Question 1.5]{Conant2018_no_intermediate}.


\bibliographystyle{alpha}
\bibliography{main}

\end{document}